\documentclass[12pt]{amsart}

\usepackage[utf8]{inputenc}
\usepackage{mathtools,empheq}
\usepackage{amsfonts}
\usepackage{amssymb}
\usepackage{amsmath}
\usepackage{amsthm}
\usepackage{verbatim}
\usepackage{esint} 
\usepackage{booktabs}
\usepackage{color}
\usepackage{xcolor} 	

\usepackage{bm} 	   	
\usepackage{relsize}   	
\usepackage{mathrsfs}  	
\usepackage{cancel}   	
\usepackage{esint}    	
\usepackage{xcolor} 	
\usepackage{cases}

\usepackage[pagebackref]{hyperref}
\hypersetup{
	colorlinks,
	citecolor=blue,
	filecolor=blue,
	linkcolor=blue,
	urlcolor=blue
}

\def\XXint#1#2#3{{\setbox0=\hbox{$#1{#2#3}{\int}$ }
\vcenter{\hbox{$#2#3$ }}\kern-.6\wd0}}
\def\lbb{\mbox{$\,$---}}
\def\Yint#1{\mathchoice
{\XXint\displaystyle\textstyle{#1}}%
{\XXint\textstyle\scriptstyle{#1}}%
{\XXint\scriptstyle\scriptscriptstyle{#1}}%
{\XXint\scriptscriptstyle\scriptscriptstyle{#1}}%
\!\iint}
\def\fiint{\Yint\lbb}

\normalbaselines
\DeclareMathOperator*{\diam}{diam}
\DeclareMathOperator*{\dist}{dist}

\DeclareMathOperator{\Div}{div}

\DeclareMathOperator{\Average}{Av}
\DeclareMathOperator{\Lip}{Lip}

\newcommand{\ColorWord}[2]{\color{#1} #2 \color{black} }
\newcommand{\WORD}[1]{\quad \text{#1} \quad}
\newcommand{\Z}{{\mathbb Z}}
\newcommand{\BMO}{{\rm BMO}}
\newcommand{\R}{{\mathbb R}}

\newcommand{\EPS}{\varepsilon}
\newcommand{\EMP}{\emptyset}
\newcommand{\HS}{H\hspace{-0.6mm}S}

\renewcommand{\div}{\mathrm{div}}

\newcommand{\Base}{\mathcal{O}}
\newcommand{\D}{\ensuremath{\mathbb{D}}}

\newcommand{\dx}{\ensuremath{\, \mathrm{d} x}}

\newcommand{\dt}{\ensuremath{\, \mathrm{d} t}}

\newcommand{\dtau}{\ensuremath{\, \mathrm{d} \tau}}

\numberwithin{equation}{section}

\theoremstyle{plain}

\newtheorem{theorem}[equation]{Theorem}
\newcommand{\reftheorem}[1]{\emph{\ColorWord{blue}{Theorem} \ref{#1}}}

\newtheorem{lemma}[equation]{Lemma}
\newcommand{\reflemma}[1]{\emph{\ColorWord{blue}{Lemma} \ref{#1}}}

\newtheorem{proposition}[equation]{Proposition}
\newcommand{\refproposition}[1]{\emph{\ColorWord{blue}{Proposition} \ref{#1}}}

\newtheorem{corollary}[equation]{Corollary}

\theoremstyle{definition}

\newtheorem{definition}[equation]{Definition}

\theoremstyle{remark}

\newtheorem{remark}[equation]{Remark}

\title[Dirichlet and Regularity Problems for Parabolic PDEs]{A relation between the Dirichlet  and the Regularity problem for Parabolic equations}
\author{Martin Dindo\v{s}}
\author{Erika Nystr\"om}

\begin{document}
\maketitle




\begin{abstract}
    We study the relationship between the Dirichlet and Regularity problem for parabolic operators
        of the form \( L = \Div(A\nabla\cdot) - \partial_t \) on cylindrical domains
        \( \Omega = \Base \times \R  \), 
        where the base \( \Base \subset \R^{n} \) is a uniform domain with $n-1$-Ahlfors regular boundary (and for one result a Lipschitz domain) in the spatial variables.
    
In the paper we answer the question when the solvability of the $L^p$ Regularity problem for $L$ (denoted by 
 \( (R_L)_{p} \)) can be deduced from the solvability of the \( L^{p'} \) Dirichlet problem for the adjoint operator $L^*$ (denoted \( (D_L^*)_{p'} \)).
  We show that this holds if for at least one $q\in(1,\infty)$ the problem \( (R_L)_{q} \) is solvable.  That is, we establish a duality/dichotomy result: Dirichlet solvability implies Regularity solvability in the dual $L^p$ range, or the Regularity problem is not solvable in any $L^p$. 
 
 Results like these were only known in the elliptic settings (Kenig-Pipher (1993) and 
Shen (2006)) but are new for parabolic PDEs. Our result is one of the key components needed for the recent advancement of Dindo\v{s}, Li and Pipher in understanding solvability of the Regularity problem for operators whose coefficients satisfy certain natural Carleson condition (called also DKP-condition in the elliptic case). 
\end{abstract}

\tableofcontents
\newpage


\section{Introduction}
In this paper we study the relationship between the parabolic Dirichlet and Regularity problems 
    on parabolic cylinders \( \Omega = \Base \times \R \) 
    for the operators of the form
\begin{equation}\label{paro}
 L = \Div(A\nabla \cdot) - \partial_t ,
 \end{equation}
    where \( A = A(X,t) \) 
    is bounded, measurable and uniformly elliptic.\medskip
    
To be precise, for any \( 1<p<\infty \), 
    we wish to show that if the \( L^{p'} \) Dirichlet problem is solvable  (which we denote by \( (D_{L^*})_{p'} \)),
     then the \( L^p \) regularity problem 
     is solvable (denoted by \( (R_L)_p \)) (see  Definitions \ref{def:Dirichlet problem} and \ref{def:regularity problem}), provided this problem is solvable for at least one $q\in (1,\infty)$.
     Note that here \( (D_{L^*})_{p'} \) denotes the solvability for the adjoint 
\[ L^* = \Div(A^T\nabla \cdot) + \partial_t, \]
    which is a backward in time parabolic operator.
Through a change of variables \( v(X,t) = u(X,-t) \), and $\tilde{A}(X,t)=A^T(X,-t)$
    we can see that the solvability of \( (D_{L^*})_{p'} \) is equivalent to solvability of \( (D_{\tilde{L}})_{p'} \)
    where 
\[ \tilde{L}v \coloneqq \Div(\tilde{A}\nabla v)-\partial_tv=0, \]
  is  the usual time-forward parabolic PDE.\medskip

It is important to establish a result of this type for several reasons. 
First, there is the inspiration provided by the elliptic theory:  in particular, the recent breakthroughs in solving the elliptic Regularity problem for operators satisfying certain natural Carleson condition (also called the DKP-condition) in \cite{DHP} and \cite{MPT}. The analogous elliptic duality/dichotomy result has been instrumental in the study of various classes of operators, and the extension of elliptic theory to the parabolic setting follows a natural and historical path.
Second, the parabolic result established here has already played a key role in the study of parabolic Regularity problems.
One of this paper's co-authors has been involved in efforts to solve the Regularity problem for operators satisfying the DKP-condition, which
was achieved in the recent preprint \cite{DLP}.  Our main theorem of this paper was a  {\it critical component} of the strategy of \cite{DLP}. Further, a recent result \cite{DU} showing that solvability of the parabolic Regularity problem is stable under 
large and small Carleson perturbations of coefficients also requires the main result of the present paper.
Finally, our result opens an avenue to establishing solvability of the parabolic Regularity problem for other classes of well studied
and natural parabolic operators once the solvability of the Dirichlet problem is known for data in some $L^p$. For example, for an operator in divergence form whose coefficients are independent of $x_n$-direction (the direction transversal to the boundary on the domain $\Omega=\R^n_+\times\R$), the solvability of the parabolic Dirichlet problem has been
obtained in \cite{AEN}. Only in the elliptic case is the Regularity problem for such operators known to be solvable \cite{HKMP} and it relied on the corresponding elliptic duality/dichotomy result.

\medskip

The elliptic version of the result we present was fully formulated by Shen in  \cite{She06}    on bounded Lipschitz domains. It uses substantially an earlier result of Kenig and Pipher \cite{KP} where it was been established that (\( (R_L)_q \)) implies (\( (R_L)_p \))
for all $1<p<q+\varepsilon$ for some small $\varepsilon>0$.

The converse, that  \( (R_L)_p \) implies \( (D_{L^*})_{p'} \) 
    has already been shown in both the elliptic \cite{DK,KP} and parabolic \cite{DD} settings,
    with the domain in \cite{DD} being a \( \Lip(1,1/2) \) domain.\medskip

Thus, having established the motivation for seeking this result, let us outline the main novelty in our arguments that distinguish it from the earlier elliptic approach.

Let us recall that the main feature that distinguishes parabolic PDEs from the elliptic ones is the presence of the time variable that enjoys different scaling than the remaining spatial variables. This is due to the fact that the second derivative $\nabla^2 u$ behaves like $\partial_t u$ (as follows from the PDE) and hence the first spatial derivative 
$\nabla u$ scales like a half time-derivative $D^{1/2}_tu$.

In particular, the consequence of this fact is then that the parabolic formulation of the Regularity problem requires boundary data with a half-time derivative and full spatial (tangential) gradient (c.f. \eqref{defDPpara}).

This simple observation gives rise to multiple problems we needed to overcome. The first  (and the most difficult problem) was to build a new parabolic version of the end-point ($p=1$) of parabolic Hardy-Sobolev space, one that respects such a scaling but that can also be (real)-interpolated against the parabolic Sobolev spaces (for $p>1$). 
We achieve this objective in section 6. This is by far the longest part of the paper as we had to build the space piece by piece from atoms that respect the desired scaling. The atoms have a free parameter $\beta\in (1,\infty]$, thus potentially obtaining not one $p=1$ Hardy-Sobolev space but a whole scale of them indexed by $\beta$. As in the elliptic case we need to establish that for each $\beta$ we obtain the same space (where a different choice of $\beta$ only changes the norm to an equivalent one). Key for this is the version of the classical \lq\lq Calder\'on-Zygmund lemma" (Proposition \ref{CZ}) but now the presence of the half-time derivative $D^{1/2}_t$ significantly complicates the matter.  In this proposition it is unrealistic to expect for the decomposition 
$f=g+\sum_i b_i$ (where $b_i$ are multiples of atoms) to have $D^{1/2}_tg\in L^\infty$ (even though we expect and obtain $\nabla g\in L^\infty$), as the mapping properties of the half-derivative operator imply that we should (at best) only expect $D^{1/2}_tg\in BMO$. 

As it turns out this opens up another rabbit hole (that we had to dive into), as proving that a half-derivative - a rather complicated non-local object -  belongs to BMO is quite a challenge. We needed to establish a new version of this condition (originally due to Strichartz \cite{S80} and adapted to parabolic case in \cite{DDH}) in Theorem \ref{T:equiv-new}. Our condition is different from the one established in \cite{DDH} and it is more convenient to work with in the light of the new characterization of Trebel-Lizorkin spaces established in \cite{YYZ}.
Finally, after Proposition \ref{CZ} is established we obtain both equivalence of the newly defined spaces for different parameters $\beta$ as well as the desired real interpolation properties using the $K$-method.\medskip

Our second important challenge was to adapt our argument to work both on domains that are bounded or unbounded in space (they are always unbounded in time variable). In section 5.2 we present a rather sophisticated localization arguments that show that key challenge are the estimates of non-tangential maximal function near support of the boundary data as \lq\lq away parts" enjoy better bounds due to exponential decay. The paper \cite{DLP} which uses our result exploits this decomposition even further.\medskip

The original formulation of the parabolic Regularity problem is due to Fabes and Rivi\`ere \cite{FR} who formulated it for $C^1$ cylinders which was then re-formulated to Lipschitz cylinders by Brown \cite{Bro87, Bro89}. Lipschitz cylinders are domains of the form $\Omega=\mathcal{O}\times [0,T]$ where $\mathcal O\subset\mathbb {R}^n$ is a bounded Lipschitz domain  and time variable runs over a bounded interval $[0,T]$. 
Subsequent authors such as Mitrea \cite{M}, Nystr\'om \cite{Nys06} or 
Castro-Rodr\'iguez-L\'opez-Staubach \cite{CRS} have dealt with this issue in a similar way.

What Brown has done was to consider solutions on $\Omega$ with initial data $u=0$ at $t=0$ and naturally extending $u$ by zero for all $t<0$. Then a definition of half-derivative using fractional integrals $I_\sigma$ 
for functions $f\in C^\infty(-\infty,T)$ is used to give meaning to $D^{1/2}_t f$ and $D^{1/2}_t u$. 

A variant of this approach (which we present here) is to consider the parabolic PDE on on infinite cylinder 
$\Omega=\mathcal O\times \mathbb R$. This can be always be achieved by extending the coefficients by setting
$A(X,t)=A(X,T)$ for $t>T$ and $A(X,t)=A(X,0)$ for $t<0$. Since the parabolic PDE has a defined direction of time the solution on the infinite cylinder $\Omega=\mathcal O\times \mathbb R$ with boundary data $u\big|_{\partial\Omega}=0$ for $t<0$ will coincide with the solution on the finite cylinder $\mathcal O\times [0,T]$ with zero initial data.\vglue1mm

Having a well-defined half-derivative the authors \cite{Bro89,CRS,M, Nys06} then define the parabolic Regularity problem by asking for
\begin{equation}\label{defDPpara}
\|N[\nabla u]\|_{L^p(\partial\Omega)}+\|N[D^{1/2}_tu]\|_{L^p(\partial\Omega)}\le C(\|\nabla_T f\|_{L^p(\partial\Omega)}+\|D^{1/2}_t f\|_{L^p(\partial\Omega)}).
\end{equation}
Some authors also ask for control of $\|N[D^{1/2}_tH_tu]\|_{L^p}$, where $H_t$ is the one-dimensional Hilbert transform in the $t$-variable.

This approach is reasonable, if the domain is not time-varying, i.e. of the form $\mathcal O\times (t_0,t_1)$ but 
runs immediately into an obvious issue when this is no longer the case, as any reasonable definition of half-derivative requires $u(X,t)$ to be defined for all $t\in\mathbb R$ or at least on a half-line.\vglue1mm

This is the reason we restrict ourselves to domains of the form $\mathcal O\times \mathbb R$ here to avoid further complications connected with dealing with this non-local term. In particular, in \cite{DLP}, that uses our result, only  domains of this type are considered.\medskip

Finally, we are ready to state our main result which we split into two theorems as there are slightly different assumptions on the domain in each case.
\begin{theorem}\label{thm:Dp_implies_Rq} Let $\mathcal O$ be a uniform domain (i.e. it satisfies interior Corkscrew and Harnack chain conditions) with $n-1$-Ahlfors regular boundary (that is $\mathcal H^{n-1}(B(x,r)\cap \partial\mathcal O)\approx r^{n-1}$)
 and $L$ be a parabolic operator of the form \eqref{paro} on the domain $\Omega=\mathcal O\times \mathbb R$.

    Suppose that the \( L^{p'} \) Dirichlet problem for the adjoint operator, \( (D_{L^*})_{p'} \), is solvable,
        and that the regularity problem \( (R_L)_q \) is solvable for some \( 1<q<p \).
    Then \( (R_L)_p \) is also solvable.
\end{theorem}

\begin{theorem}\label{thm:Rp_extrapolation}
    Assume that \( \Base \) is Lipschitz (either bounded or unbounded) and $L$ be a parabolic operator of the form \eqref{paro} on the domain $\Omega=\mathcal O\times \mathbb R$.
    
Suppose that \( (R_L)_p \) is solvable for some \( 1<p<\infty \).
    Then \( (R_L)_q \) is solvable for \( 1<q<p \).
\end{theorem}

\reftheorem{thm:Dp_implies_Rq} is the parabolic version of the result of \cite{She06}, while \reftheorem{thm:Rp_extrapolation} is an analogue of the elliptic result in \cite{KP95}.

By combining these two theorems we can obtain the following: 
\begin{theorem}\label{thm:main}
    Assume that \( \Base \) is Lipschitz (either bounded or unbounded) and $L$ be a parabolic operator of the form \eqref{paro} on the domain $\Omega=\mathcal O\times \mathbb R$.
    
Suppose that \( (D_{L^*})_{p'} \) is solvable for some $p\in (1,\infty)$. 
   Then either \( (R_L)_p \) is solvable or \( (R_L)_q \) is not solvable for any \( 1<q<\infty \).
\end{theorem}

In summary, thanks to \reftheorem{thm:main} and \cite{DD} we have the following two implications
\[ (R_L)_p \implies (D_{L^*})_{p'} \WORD{and} (D_{L^*})_{p'} + (R_L)_q \implies (R_L)_p, \]
for parabolic operators on Lipschitz cylinders. \medskip

We have left open some further research directions, \reftheorem{thm:Rp_extrapolation} has stronger assumption on the domain $\mathcal O$ than the other result and thus it might be possible to improve this result further beyond the Lipschitz cylinders. What limits us currently are interpolation results such as \refproposition{IHS2} which only works on boundaries that locally look like $\R^n$. 

Another natural direction would be to consider time varying domains (restricting initially to domains that locally look like a graph of \( \Lip(1,1/2) \) functions).
\medskip

The paper is set up as follows. Section \ref{S:basics} 
    contains definitions needed for the rest of the paper and Section \ref{S3} summarises known parabolic results we needs for our proofs.
 \reftheorem{thm:Dp_implies_Rq} is proven first, motivated by the method from \cite{She06}. 
    The two cases \( \diam \Base = \infty \)  and
 \( R \coloneqq \diam\Base < \infty \) are considered separately.
 
\reftheorem{thm:Rp_extrapolation} follows by real interpolation against a suitable endpoint space which we have already introduced. With the interpolation result in hand, the final piece missing is the establishment of an endpoint bound for \( \|N[\nabla u]\|_{L^1(\partial\Omega)} \) which closely follows its elliptic version as in \cite{DK}.

\section{Basic definitions}\label{S:basics}

\subsection{Domain}
In general our domain will be a cylinder \( \Omega = \Base \times \R \),
    where the base \( \Base \) is a  1-sided non-tangentially accessible domain (also called uniform domain) with Ahlfors regular boundary of dimension $n-1$. As the precise definition does not play any significant role here we omit the details and refer the reader to \cite{D} which has the same settings as here and contains all definitions in great detail (in particular, the existence of interior Harnack chains and corkscrew balls for any boundary point plus the Ahlfors condition). 
  
We denote points in \( \Omega \) by \( (X,t) \) or also \( (\xi,\tau) \), where \( X,\,\xi \in \Base \) and \( t,\,\tau \in \R \),
    and points on the boundary by \( (P,t) \in \partial \Omega = \partial\Base \times \R \).
The space \( \Omega \) is equipped with the norm \( \|(\xi,\tau)\| = \rho \),
    where $\rho$ is the positive solution  to the following equation
\begin{equation}
	\label{E:par-norm}
	\frac{|\xi|^2}{\rho^2} + \frac{\tau^2}{\rho^4} = 1.
\end{equation}
One can easily show that $\|(\xi,\tau)\| \sim |\xi| + |\tau|^{1/2}$ and 
that this norm scales correctly according to the parabolic nature of the 
PDE.

Next we introduce notation for balls, cubes and cylinders
\begin{itemize}
    \item \(B_r(X,t):=\{(Y,s) \in \Omega; \|(Y,s)-(X,t)\|\leq r\}\)$\,\,\qquad$ (parabolic ball),
    \item \(B_r(X):=\{Y\in \mathcal{O}; |Y-X|\leq r\}\)$\qquad\qquad\qquad\qquad$ (ball in space only),
    \item \( C_r(X,t) := B_r(X) \times (t-r^2,t+r^2) \)$\qquad\qquad\qquad\,\,\,\,$ (parabolic cylinder),
    \item \(Q_r(X,t):=\{(Y,s)\in \Omega; |Y_i-X_i|\leq r, |s-t|\leq r^2\}\)$\,\,\,$(parabolic cube),
    \item \(\Delta_r(P,t):=\partial\Omega\cap B_r(P,t) \)$\,\,\quad\qquad\qquad\qquad\qquad\qquad$ (boundary ball),
    \item \(T(\Delta_r(P,t)):=\Omega\cap B_r(P,t) \)$\,\,\quad\qquad\qquad\qquad\qquad\quad$ (Carleson region).
\end{itemize}

Occasionally, we drop dependence of these sets on the point $(X,t)$ and $r$ and just write $B,\,C,\,Q$ respectively. When doing this, the notation $r(B)$ or $r(Q)$ will denote the \lq\lq radius" $r$ of these sets. We use the similar convention also for $\Delta$ and $T(\Delta)$. 
When writing $2B$ or $5\Delta$, it means enlarged sets (with same center point but $2\times$ or $5\times$ the radius).
Next, given \( (X,t) \in \Omega \) we let \( \delta(X,t) = \delta(X) = \dist(X,\Base) \)
    denotes the parabolic distance to the boundary.     

Note that because \( \Base \) is a CAD we have that for \( P \in \Base \) and \( \diam \Base > r>0 \)
    the existence of a corkscrew point \( A_r(P) \).
Thus for \( \Delta = \Delta(P,t;r) \subset \Omega \) with \( r < \diam \Base \) 
    we define \( V^+(\Delta) = (A_r(P),t+100r^2) \) and \( V^-(\Delta) = (A_r(P),t-100r^2) \)
    as the \lq\lq forward in time" and \lq\lq backward in time" corkscrew point respectively.

Given an aperture \( a > 0 \), it is customary to define non-tangential approach regions (\lq\lq cones") as follows:
\[\tilde{\Gamma}_a(P,t) \coloneqq 
        \{ (X,s) \in \Omega : \|(X,s) - (P,t)\| < (1+a)\delta(X) \}.\]
However such cones lack one desired property, namely that the set (which can be thought of as the reverse cone originating from $(X,t)$)
$$\{(P,\tau)\in\partial\Omega:\, (X,t)\in\tilde\Gamma_a(P,\tau)\}$$
always has the surface area comparable to the $\delta(X,t)^{n+1}$. To fix this issue we do as follows. Let
\begin{itemize}
    \item \( \gamma_a(P) \coloneqq \{ X \in \Base : |X-P| < (1+a)\delta(X) \} \),
    \item \( S_a(X) \coloneqq \{ P \in \partial\Base : X \in \gamma_a(P) \} \),
    \item \( \tilde{S}_a(X) \coloneqq \bigcup_{P \in S_a(X)} ,
        \{ P' \in \partial\Base : |P'-P| < a\delta(X) \} \),
    \item \( \tilde{\gamma}_a(P) \coloneqq \{ X \in \Omega : P \in \tilde{S}_a(X) \} \),
    \item \( \Gamma_a(P,t) \coloneqq \{ (X,s) \in \Omega : X \in \tilde{\gamma}_a(P), 
        \; |s-t| < a\delta(X)^2 \} \).
\end{itemize}
Note that $\Gamma_a(P,t)$ defined in the last line is our version of the non-tangential approach region and
    is equivalent to $\tilde{\Gamma}_a(P,t)$ since for any $a>0$ there exists $a_1,a_2>0$ such that
$$   \Gamma_{a_1}(P,t)   \subset \tilde{\Gamma}_a(P,t) \subset {\Gamma}_{a_2}(P,t)$$
and vice versa. However, as the set $\tilde{S}_a(X)$ can be written as a union of boundary balls of radius $a\delta(X)$ it follows that the size of this set is always comparable to $\delta(X,t)^{n+1}$ as desired.
For the rest of the paper we will let \( a=1 \) and drop the subscript next to $\Gamma$.
    
\begin{definition}
    Given a function \( u : \Omega \to \R \), the non-tangential maximal function is given by
    \[ N_{\sup}[u](P,t) \coloneqq \sup_{\Gamma(P,t)} |u|. \]
    Since \( \nabla u \) for example, is not defined pointwise it is customary to define a version of non-tangential maximal function using \( L^2 \) averages as follows:
    \[ N[u](P,t) \coloneqq \sup_{\Gamma(P,t)} \Average^2[u], 
        \quad \Average^2[u](X,s) \coloneqq \Big( \fint_{C_{\delta(X)/4}(X,s)} |u|^2 \Big)^{1/2}. \]
Occasionally, we will also consider truncated non-tangential cones $\Gamma^r(P,t):=\Gamma(P,t)\cap B_r(P,t)$
and corresponding truncated non-tangential maximal function $N^r$.        
\end{definition}
We are now ready to define the Dirichlet problem
\begin{definition}\label{def:Dirichlet problem}
	Let \( 1 < p < \infty \)
	We say that the \( L^p \) \emph{Dirichlet problem} for \( L \) \( (D_L)_p \) is solvable 
		if for each continuous \( f \in L^p(\partial\Omega) \) the solution to
	\[ \begin{cases}
		Lu = 0, \quad \text{in } \Omega
        \\
		u = f, \quad \text{on } \partial\Omega
	\end{cases} \]
		satisfies
	\[ \| N[u] \|_{L^p(\partial\Omega)} \lesssim \| f \|_{L^p(\partial\Omega)}. \]
\end{definition}

In the definition above it does not matter which version of non-tangential maximal function is used. When it is not explicitly stated, we will always mean $N$ defined using the $L^2$ averages.

\subsection{Half time derivative and Parabolic Sobolev Space on {$\mathbb R^n$}.}
Here we consider an appropriate function space for our boundary data on 
$\mathbb R^n=\mathbb R^{n-1}\times\mathbb R$.

As a rule of thumb, one derivative in time behaves as two 
derivatives in spatial variables. Thus if we impose boundary data having one spatial derivative,  the correct order of the time derivative is $1/2$, as we would like to respect the correct parabolic scaling.

This problem has been studied previously in~\cite{HL96,HL99,Nys06}, who 
have followed~\cite{FJ68} in defining the homogeneous parabolic Sobolev 
space $\dot{L}^p_{1,1/2}$ in the following way. 

\begin{definition}%
	\label{D:paraSob}
	The \textit{homogeneous parabolic Sobolev space} 
$\dot{L}^p_{1,1/2}(\R^n)$, for $1 < p < \infty$, is defined to consist of 
an equivalence class of functions $f$ with distributional derivatives 
satisfying $\|f\|_{\dot{L}^p_{1,1/2}(\R^n)} < \infty$, where
	\begin{equation}
		\|f\|_{\dot{L}^p_{1,1/2}(\R^n)} = \|\D f\|_{L^p(\R^n)}
	\end{equation}
	and
	\begin{equation}
		(\D f)\,\widehat{\,}\,(\xi,\tau) := \|(\xi,\tau)\| 
\widehat{f}(\xi,\tau),\label{e12}
	\end{equation}
with $\|(\xi,\tau)\|$ on $\R^{n-1} \times \R$ defined by \eqref{E:par-norm}.
\end{definition}

In addition, following~\cite{FR67}, we define a parabolic half-order time 
derivative by
\begin{equation}
	(\D_n f)\,\widehat{\,}\,(\xi,\tau) := \frac{\tau}{\|(\xi,\tau) \|} 
\widehat{f}(\xi,\tau).
\end{equation}

If $0 < \alpha \leq 2$, then for $g \in C^\infty_c(\R)$ the 
\textit{one-dimensional fractional differentiation operators} $D_\alpha$ 
are defined by
\begin{equation}
	(D^\alpha g)\,\widehat{\,}\,(\tau) := |\tau|^\alpha 
\widehat{g}(\tau).
\end{equation}
It is also well known that if $0 < \alpha < 1$ then
\begin{equation}
	D^\alpha g(s) = c\int_{\R} \frac{g(s) - g(\tau)}{|s-\tau|^{1 + 
\alpha}} \dtau
\end{equation}
whenever $s \in \R$.
If $h(x,t) \in C^\infty_c(\R^n)$ then by $D^\alpha_t h: \R^n \rightarrow 
\R$ we mean the function $D^\alpha h(x, \cdot)$ defined a.e.\ for each 
fixed $x \in \R^{n-1}$. When $\alpha=1/2$ we have
\[ D^{1/2}_tD^{1/2}_t = -H_t \partial_t, \]
where \( H_t \) is the Hilbert transform taken in the time variable.

We record the connection between $\D$, $\D_n$ and $D_t^{1/2}$. By 
~\cite{DD,FR66,FR67} we have that
\begin{align}\label{norms}
	\|\D f\|_{L^p(\R^n)} \sim  \|\D_n f\|_{L^p(\R^n)}+ \|\nabla 
f\|_{L^p(\R^n)}\sim  \|D_t^{1/2} f\|_{L^p(\R^n)}+\|\nabla f\|_{L^p(\R^n)},
\end{align}
for all $1<p<\infty$, as well as at the $\BMO$ end-point:
\begin{align}\label{normsbmo}
	\|\D f\|_{\BMO(\R^n)} \sim  \|\D_n f\|_{\BMO(\R^n)}+ \|\nabla 
f\|_{\BMO(\R^n)}\sim  \|D_t^{1/2} f\|_{\BMO(\R^n)}+\|\nabla f\|_{\BMO(\R^n)}.
\end{align}
Here $\nabla$ denotes the usual gradient in the 
variables $x\in\mathbb R^{n-1}$ and $\BMO$ denotes the usual space of bounded mean oscillations (with respect to parabolic balls). \medskip

When $p=1$ we no longer have \eqref{norms} and hence the norm of $\dot{L}^1_{1,1/2}(\R^n)$ might not be equivalent to 
$\|\nabla f\|_{L^1}+\|D^{1/2}_t f\|_{L^1}$. We therefore define the space
\begin{equation}\label{W}
\dot{W}^{p}_{1,1/2}(\R^n)=\{f\in L^1_{loc}(\R^n): \|\nabla f\|_{L^p(\R^n)}+\|D^{1/2}_t f\|_{L^p(\R^n)}<\infty\},
\end{equation}
with norm $\|\nabla f\|_{L^p}+\|D^{1/2}_t f\|_{L^p}$. Modulo constants when $p\in (1,\infty)$ we have
$\dot{W}^{p}_{1,1/2}(\R^n)=\dot{L}^{p}_{1,1/2}(\R^n)$ and the norms are comparable, but when $p=1,\infty$ the spaces might not coincide. These spaces will be useful for interpolation. We will also need a weaker end-point space than $\dot{W}^{\infty}_{1,1/2}(\R^n)$. We denote by $\dot{W}^{\infty,{\BMO}}_{1,1/2}(\R^n)$ the space such that 
    $\nabla f\in L^\infty(\R^n)$ and $D^{1/2}_t f\in \BMO(\R^n)$.

\subsection{Connections to Triebel-Lizorkin and Haj\l{}asz-Sobolev spaces.}%
\label{S:paraSobolev}
Here and in what follows we will use notation $X=(x,t)$ to denote points in $\R^n$ with components  $x\in\R^{n-1}$ and $t\in\R$. 
\vglue2mm
Let us recall the definition of Triebel-Lizorkin spaces 
$F^{s}_{p,q}$. The fact that our operator $\D$ is defined by \eqref{e12} implies that it fits naturally to the settings of these spaces for $q=2$, but we need to account for the parabolic scaling. 

\begin{definition} Let $s\in (0,\infty)$ and $p\in (0,\infty]$. Let $\varphi\in\mathcal S(\R^n)$ such that
$$\mbox{\rm supp }\hat\varphi\subset \{(\xi,\tau): 1/2\le \|(\xi,\tau)\|\le 3/2\},$$
and 
$$|\hat\varphi(\xi,\tau)|\ge c_0>0\mbox{ for } 3/5\le \|(\xi,\tau)\|\le 5/3.$$
Let $\varphi_\rho(x,t)$  denote the usual parabolic dilation by $\rho$, that is
	\[
		\varphi_\rho(x,t) = \rho^{-(n+1)}\varphi(x/\rho, t/\rho^2).
	\]
The homogeneous {\rm parabolic Triebel-Lizorkin space} $\dot{F}^s_{p,2}(\R^n)$ is defined as the collection of all $f\in\mathcal S'(\R^n)$ such that $\|f\|_{\dot{F}^s_{p,2}(\R^n)}<\infty$ where
\begin{equation}
\|f\|_{\dot{F}^s_{p,2}(\R^n)}=\left\|\left(\sum_{k\in\mathbb Z}4^{ks}|\varphi_{2^{-k}}* f|^2\right)^{1/2}\right\|_{L^p(\R^n)},
\end{equation}
with the usual modification when $p=\infty$.
\end{definition}
By the Littlewood-Paley theorem when $M_k f=\varphi_{2^{-k}}* f$  then for $1<p<\infty$
$$\|f\|_{L^p(\R^n)}\approx \left\| (M_k f)_{\ell^2(\mathbb Z)}\right\|_{L^p(\R^n)}=\|f\|_{\dot{F}^0_{p,2}(\R^n)}.$$
Hence
$$\|f\|_{\dot{L}^p_{1,1/2}(\R^n)}\approx \left\| (2^kM_k f)_{\ell^2(\mathbb Z)}\right\|_{L^p(\R^n)}=\|f\|_{\dot{F}^1_{p,2}(\R^n)}.$$
Note that is is only true when parabolic scaling is applied, in the usual Euclidean metric we have
$\dot{L}^p_{1/2}(\R)=\dot{F}^{1/2}_{p,2}(\R)$ and $\dot{L}^p_{1}(\R^{n-1})=\dot{F}^{1}_{p,2}(\R^{n-1})$
(see \cite{MPS}).

The paper \cite{YYZ} gave a different characterisation of these spaces using averages. Again we have to account for parabolic scaling, but the argument adapts easily. Let us first define

\begin{equation}\label{sqfn-disc}
S_{s,disc}(f)(x,t)=\left(\sum_{k\in\mathbb Z} 4^{ks}\left(f(x,t)-\fint_{B_{2^k}(x,t)}f\right)^2\right)^{1/2}.
\end{equation}
Here $B_\alpha(x,t)$ is the parabolic ball centred at $(x,t)$ with radius $\alpha$. By Theorem 1.1 of \cite{YYZ}
$\|S_{s,disc}(f)\|_{L^p}\sim \|f\|_{\dot{F}^s_{p,2}}$ for all $1<p<\infty$. For our further considerations we would prefer continuous version of the square function \eqref{sqfn-disc}. It is not hard to see that by averaging over set of balls of radius $2^{\tau k}$, $1\le\tau\le 2$ we may replace it by
\begin{equation}\label{sqfn}
S_{s}(f)(x,t)=\left(\int_0^\infty \alpha^{-2s-1}\left(f(x,t)-\fint_{B_\alpha(x,t)}f\right)^2d\alpha\right)^{1/2}.
\end{equation}
In particular, for $s=1$ we have $\|f\|_{\dot{L}^p_{1,1/2}(\R^n)}\approx \|S_1(f)\|_{L^p(\R^n)}$ for $1<p<\infty$.

We develop an interesting connection between this square function and new characterisation of the fact that $\D f \in \BMO(\R^n)$. Recall that \cite{DDH} (Theorem 2.3) showed the following:
\begin{theorem}%
	\label{T:equiv}
	Let $f : \R^{n-1} \times \R \to \R$ and $f \in {\rm Lip}(1,1/2)$ then the following conditions are equivalent:
	\begin{itemize}
		\item $\D f \in\BMO(\R^n)$

		\item\label{I:av}
		\begin{equation}
			\label{E:Dt:riv}
			\sup_{B_r} \frac{1}{|B_r|}\int\limits_{B_r}\int\limits_{\|(y,s)\| \leq r}
			\frac{|f(x+y,t+s) - 2f(x,t) + f(x-y,t-s)|^2}{\|(y,s)\|^{n+3}} \,dy\,ds\,dx\,dt =  B_{1} < \infty,
		\end{equation}
\end{itemize}
	\begin{equation}
		\label{E:equiv:norms}
		\|\D f \|_\BMO^2
		\sim B_{1}.
	\end{equation}
\end{theorem}

Now in the light of \eqref{sqfn} we need to obtain a new version of the condition \eqref{E:Dt:riv} where 
the absolute value is moved to the outside of the interior integral.  A method how to establish such conditions was developed by Strichartz \cite{S80} and our paper \cite{DDH} adapts it to parabolic settings. We claim the following:
\begin{theorem}%
	\label{T:equiv-new}
	Let $f : \R^{n-1} \times \R \to \R$ and $f \in {\rm Lip}(1,1/2)$ then the following conditions are equivalent:
	\begin{itemize}
		\item $\D f \in\BMO(\R^n)$
\item

			\begin{equation}\label{E:av:half-new}
				\sup_{B_r} \frac{1}{|B_r|}\int_{B_r}\int_{0}^{r}
				\alpha^{-3}\left|f(x,t)-\fint_{B_\alpha(x,t)}f\,\right|^2\,d\alpha\dt\dx = B_2 < \infty.
			\end{equation}
\end{itemize}
Moreover, $\|\D f\|_\BMO^2
		\sim B_{2}$ and $r$ is the diameter of the parabolic ball $B_r$. The statement also holds with parabolic cubes $Q_r$ replacing the parabolic balls $B_r$.
\end{theorem}
Perhaps the two conditions \eqref{E:Dt:riv} and \eqref{E:av:half-new} do not immediately look close, but after writing the interior integral in polar coordinates  will reveal that we are taking averages over the shells $\|(y,s)\|=const$
with the same scaling factor $\|(y,s)\|^{-3}$ as in \eqref{E:av:half-new}. The fact that the absolute value in \eqref{E:av:half-new} is on the outside of the most interior integral is the innovation \cite{YYZ} has introduced.

\begin{proof}  Let us recall the method of Strichartz adapted to the parabolic setting in \cite{DDH}. Write $F=\D f$ where $F$ is a tempered distribution.
	Let
	\begin{equation*}
		\varphi = \delta(0,0)-\chi_{\tilde{Q}_1(0,0)},
	\end{equation*}
where $Q_1(0,0)$ is a cube ball in $\R^n$. We may take a ball here instead of the cube, it will change the Fourier transform somewhat but not the gist of the argument (we prefer not to deal with Bessel's functions). Then
	\begin{equation}
		\label{E:equiv:phi:k}
		\begin{split}
			\widehat{\varphi}(\xi, \tau) &=1- \frac{1 - e^{-i\tau}}{i\tau} \prod^{n-1}_{j=1} \frac{1 - e^{-i\xi_j}}{i\xi_j}. 		\end{split}
	\end{equation}
	with~$\widehat{\varphi}(\xi, \tau) \sim |\xi|+|\tau|$ for small $\xi$ and $\tau$.
	We let
	\begin{equation*}
		\widehat{\psi} = \frac{\widehat{\varphi}(\xi, \tau)}{\|(\xi,\tau)\|},
	\end{equation*}
where $\psi_\rho$ denotes the usual parabolic dilation $\rho$ defined above.
	We may rewrite \eqref{E:av:half-new}  as
	\begin{equation}\label{I:av:grad,I:av:Dt}
		\sup_{Q_r} \frac{1}{|Q_r|} \int_{Q_r} \int_0^r
		\left(\psi_\rho * F \right)^2 \frac{\,d\rho}{\rho}\,dx\,dt
		\sim B_{2}.
	\end{equation}
	The function $\psi$  satisfy the following conditions for some $\epsilon_i > 0$
	\begin{equation}
		\label{E:equiv:estimates}
		\begin{split}
			\int \psi \dx\dt &= 0, \\
			|\psi(x,t)| &\lesssim \|(x,t)\|^{-n-1-\epsilon_1} \text{ for } \|(x,t)\| \geq a > 0, \\
			|\widehat{\psi}(\xi,\tau)| &\lesssim \|(\xi,\tau)\|^{\epsilon_2} \text{ for } \|(\xi,\tau)\| \leq 1, \\
			|\widehat{\psi}(\xi,\tau)| &\lesssim \|(\xi,\tau)\|^{-\epsilon_3} \text{ for } \|(\xi,\tau)\| \geq 1.
		\end{split}
	\end{equation}
	Therefore if $\D f = F \in \BMO(\R^n)$ then $B_{2} \lesssim \|\D\phi\|_*^2$  by~\cite[Theorem 2.1]{S80}; 
        this shows that first condition in \reftheorem{T:equiv-new} implies the second one.

	For the converse we proceed via an analogue of the proof of~\cite[Theorem 2.6]{S80}.
	Consider
	\begin{equation*}
		\hat{\theta}(\xi,\tau) = \|(\xi,\tau)\| \hat{\zeta}(\xi,\tau),
	\end{equation*}
	where $\zeta \in C^\infty_0(\R)$.
	Let $H^1_{00}$ be the dense subclass of continuous $H^1$ functions $g$ such that $g$ and all its derivatives decay rapidly, see~\cite[p.~225]{Ste70}.
	Via an analogue of~\cite[Theorem 3]{FS72}, \cite[Lemma 2.7]{S80} by assuming \eqref{I:av:grad,I:av:Dt} if $g \in H^1_{00}(\R^n)$, then 
\begin{equation}
		\label{E:equiv:duality:k}
		 \left|\int_0^\infty \iint\limits_{\R^{n-1} \times \R}
		\psi_\rho  * F(x,t) \theta_\rho * g(x,t) \,dx\,dt\frac{\,d\rho}{\rho} \right| \lesssim B^{1/2}_{2} \|g\|_{H^1}.
	\end{equation}
	Let
	\begin{equation}
		\label{E:equiv:m}
		\begin{split}
			m(\xi,\tau) &=  \int_0^\infty
			\overline{\hat{\psi} \left(-\rho\xi,-\rho^2 \tau\right)}
			\|(\xi,\tau)\| \zeta(\rho\|(\xi,\tau)\|)
			\,d\rho.
		\end{split}
	\end{equation}
	Such function $m$ is homogeneous of degree zero, smooth away from the origin and the associated Fourier multipliers $M$ is a Caldor\'on-Zygmund operator that preserves the class $H^1_{00}$ and is bounded on $H^1$.

	The non-degeneracy condition from~\cite{CT75}  is the property that $ |m(r \xi, r^2\tau)|^2$ does not vanish identically in $r$ for $(\xi,\tau) \neq (0,0)$. Hence as in \cite{DDH} we may conclude that
$(h,g) = (F,g)$ for all $g \in H^1_{00}$ where $h$ is a $\BMO$ function with $\|h\|_*^2 \lesssim B_{2}$.
	The identity $\hat{h}=\hat{F}$
	 does not need to hold at the origin wherefore $\hat{h} - \hat{F}$ may be supported at the origin and hence $F = h + p$ where $p$ is a polynomial.
	Due to the assumption $\phi\in {\rm Lip}(1,1/2)$, clearly $F$ must be a tempered distribution. Hence as in~\cite{S80} we may conclude $F = h \in \BMO(\R^n)$.  This proves the claim.
\end{proof}

Now, we shall make connection between $\dot{L}^p_{1,1/2}(\R^n)$ and the  Haj\l{}asz-Sobolev space $\dot{M}^{1,p}_d(\R^n)$. In fact, we want to consider a more general setting and work on a set $\partial\mathcal O\times\R$, where 
 $\partial\mathcal O$ is $n-1$-Ahlfors regular set with respect to the standard $n-1$-dimensional Hausdoff measure. 
 In particular this allows us to talk about the Lebesgues spaces $L^p(\partial\mathcal O)$.
Thus, when $\mathcal O=\R^{n}_+$,  we have $\partial\mathcal O=\R^{n-1}$ and therefore the discussion below applies to the space $\R^n=\R^{n-1}\times \R$ discussed above.

\begin{definition} Let $d$ be a complete metric on a measure space $(\mathcal X,\mu)$. For $p\ge 1$, $0<s\le 1$ we say that a function $f\in L^p_{loc}(\mathcal X)$ belongs to the Hajłasz–Sobolev space $M^{s,p} (\mathcal X)$ 
    if there exists a set $E$ of $\mu$-measure zero such that
\begin{equation}\label{split1bz}
|f(X)-f(Y)|\le d(X,Y)^s|[g(X)+g(Y)]\qquad \forall X,Y\in \mathcal X\setminus E.
\end{equation}
The $M^{s,p}_d(\mathcal X)$ norm of $f$ is given as the infimum of $L^p$ norms of all such $g$.
\end{definition}

In what follows we shall consider the space $M^{1,p}_d(\partial\mathcal O\times\R)$
defined with respect to the  parabolic distance function
\begin{equation}\label{PD}
d(X,X')=d((x,t),(x',t'))=\|(x-x',t-t')\|\sim |x-x'| + |t-t'|^{1/2}.
\end{equation}
The underlying measure is the $n$-dimensional Hausdorff measure on $\partial\mathcal O\times\R\subset \R^{n+1}$. Here we introduce notation that matches the convention used above, i.e., $X\in \partial\mathcal O\times\R$ has the spatial coordinate $x$ and time coordinate $t$.
\medskip

Associated with the space $\dot{M}^{1,p}$ is the so-called {\it Sobolev sharp maximal function} $f^\sharp$ defined by
\begin{equation}\label{sharp}
f^\sharp(X)=\sup_{B:x\in B}\frac1{r(B)}\fint_B|f-f_B|.
\end{equation}
Here $B$ denotes a surface ball on the set $\partial\mathcal O\times\R$ with respect to metric $d$ of radius $r(B)$, i.e., 
$$B=B_{r(B)}(Y)=\{Z\in\partial\mathcal O\times\R:d(Y,Z)<r(B)\},$$
and $f_B$ denotes the average $f_B=\fint_B f(Z)\,dZ$. \vglue1mm

Let us also introduce $M$ which is the usual (non-centred) Hardy-Littlewood parabolic maximal function
\begin{equation}\label{lHL}
M(f)(X)=\sup_{B:X\in B}\fint_B |f|\,d\mathcal H^n,
\end{equation}
where the supremum is taken over all surface parabolic balls $B$ that contain point $X$. Denote by
$$M_q(f)=M(|f|^q)^{1/q}.$$
It follows from the usual boundedness of the maximal operator (which is true since $\partial\Omega$ is an Ahlfors regular set) that $M$ is $L^p$ bounded for $p>1$ and $M_q$ for $p>q$.

Let us denote by $M^t$ and $M^x$ the variants of the maximal operator when the averages are taken only in $t$ or $x$ direction respectively. That is 
$$M^x(f)(x,t)=\sup_{\Delta:x\in \Delta}\fint_\Delta |f(y,t)|\,dy,$$
and $\Delta=\Delta(z)=\{y\in\partial\mathcal O:|y-z|<r\}$, with $M^t$ defined analogously. We also define $M^x_q$ and  $M^t_q$ in analogy to $M_q$.

As on each fibre for a fixed $t$ we have $M^x$ being bounded on $L^p(\partial\mathcal O)$ for $p>1$ we see by Fubini that $M^x$ and $M^t$ are $L^p(\partial\mathcal O\times\R)$ bounded for $p>1$ and the same is also true for $M^x_q$ and  $M^t_q$ for $p>q$.\vglue1mm

\vglue2mm

On metric spaces with doubling measure $\mu$, that is for some constant $C_s>1$
$$\mu(B(X,2r))\le C_s\mu(B(X,r)),\qquad\mbox{ for all points $X$ and $r>0$},$$
we have by \cite{KT} for all $p\ge 1$:
$$\|f\|_{\dot{M}^{1,p}_d}\sim \|f^\sharp\|_{L^p}$$
and \eqref{split1bz} holds with $g$ replaced by $f^\sharp$. In our case for $d$ defined by \eqref{PD} the volume of surface ball of radius $r$ is proportional to $r^{n+1}$ and hence the doubling constant $C_s$ can be taken to be $2^{n+1}$. As in \cite{KT} $s=\log_2 C_s$ is called the doubling dimension of our space, thus for $\partial\mathcal O\times\R$ equipped with the metric $d$ the doubling dimension is $s=n+1$.\vglue2mm

In particular,  we have on our space $\dot{M}_d^{1,p}(\partial\mathcal O\times\R)$ for a.e. $t\in\R$:
\begin{equation}\label{m1}
|f(x,t)-f(y,t)|\le |x-y|[g(x,t)+g(y,t)]\qquad\mbox {for $\mathcal H^{n-1}$-almost every }x,y\in\mathbb \partial\mathcal O,
\end{equation}
and for a.e. $x\in\partial\mathcal O$
\begin{equation}\label{m1a}
|f(x,t)-f(x,\tau)|\le |t-\tau|^{1/2}[g(x,t)+g(x,\tau)]\qquad\mbox {for almost every }t,\tau\in\mathbb R.
\end{equation}
For all $p>1$,
by Corollary 1.2 of \cite{KYZ}  when $t$ is fixed we have $\dot{L}^p_1(\R^{n-1})=\dot{F}^{1}_{p,2}(\R^{n-1})=\dot{M}_{|\cdot|}^{1,p}(\R^{n-1})$ and so \eqref{m1} is equivalent to having $\nabla f\in L^p(\R^n)$. 
Here $\dot{M}_{|\cdot|}^{1,p}$ is defined with respect to the Euclidean distance $|\cdot|$.

When $\mathcal O$ is a Lipschitz domain, the Triebel-Lizorkin space $\dot{F}^{1}_{p,2}(\partial\mathcal O)$ is still well-defined and we might take it as the definition of the space $\dot{L}^p_1(\partial\mathcal O)$. 
Thus in this case we still have $\dot{L}^p_1(\partial\mathcal O)=\dot{F}^{1}_{p,2}(\partial\mathcal O)=\dot{M}_{|\cdot|}^{1,p}(\partial\mathcal O)$ as Corollary 1.2 of \cite{KYZ} applies.  

Also since $\dot{L}^p_{1/2}(\R)=\dot{F}^{1/2}_{p,2}(\R)$, by Corollary 1.2 of \cite{KYZ}
then $\dot{L}^p_{1/2}(\R)=\dot{F}^{1/2}_{p,2}(\R)\subset \dot{F}^{1/2}_{p,\infty}(\R)=\dot{M}_{|\cdot|}^{1/2,p}(\R)$
 (here the switch to $1/2$ derivative comes due to the fact that \eqref{m1a} uses the usual Euclidean metric on $\R$).

Hence for $f\in \dot{L}^p_{1,1/2}(\partial\mathcal O\times\R)$, $p>1$ we have that both \eqref{m1} and \eqref{m1a} hold. This in turn implies that so does 
\begin{equation}\label{m1ax}
|f(x,t)-f(y,\tau)|\le d((x,t),(y,\tau))[h(x,t)+h(y,\tau)]
\end{equation}
for almost every pairs of points in $\partial\mathcal O\times\R$ with $h=M^t(g)+M^x(g)\in L^p(\partial\mathcal O\times\R)$.
It follows that 
\begin{equation}\label{m22}
\dot{L}^p_{1,1/2}(\partial\mathcal O\times\R)\subset \dot{M}_d^{1,p}(\partial\mathcal O\times\R)\quad\mbox{and}\quad \|f\|_{\dot{M}_d^{1,p}(\partial\mathcal O\times\R)}\lesssim \|f\|_{\dot{L}^p_{1,1/2}(\partial\mathcal O\times\R)},\quad p>1.
\end{equation}\vglue1mm

Consider any bounded nonnegative function supported in $B$ and let $\chi(B)=\int_B\chi$. Suppose that 
$\chi(B)\sim |B|$.  Consider
$$\left|f(X)-\frac{1}{{\chi}(B)}\int_{B}f\chi\right|=\left|\frac{1}{{\chi}(B)}\int_B (f(X)-f(Z))\chi(Z)dZ\right|
$$
$$\le \sup_{Z\in B}d(X,Z)\frac{1}{{\chi}(B)}\int_B[f^\sharp(X)+f^\sharp(Z)]\chi(Z)dZ,$$
which holds by \eqref{split1bz} and the fact that $f^\sharp$ can be used in place of $g$. Since $X\in B$ we have $\sup_{Z\in B}d(X,Z)\lesssim r(B)$ and hence 
\begin{equation}\label{sharp2}
\frac1{r(B)}\left|f(X)-\frac{1}{{\chi}(B)}\int_{B}f\chi\right|\lesssim f^\sharp(X)+\fint_Bf^\sharp,
\end{equation}
for any $X$ for which $X\in B$. \medskip

\subsection{Parabolic Regularity problem}

The previous section has defined the space $\dot{L}_{1,1/2}(\partial\mathcal O\times\R)$ in the case when $\mathcal O\subset\R^n$ is a bounded or unbounded Lipschitz domain. In particular, our definition implies that when
$$\mathcal O=\{(x',x_n):\, x_n>\phi(x')\},$$
for some Lipschitz function $\phi:\R^{n-1}\to\R$, then
$f\in \dot{L}_{1,1/2}(\partial\Omega)$ if 
\begin{equation}\label{defsp}
f\circ\pi\in \dot{L}_{1,1/2}(\R^n),\qquad\mbox{where}\qquad \pi(x',t)=(x',\phi(x'),t).
\end{equation}

\medskip

We would like to extend our definition to the case $\mathcal O$ is a uniform domain with $n-1$-Ahlfors regular boundary in a way that is consistent with the previous subsection and so that for example \eqref{m22} still holds.\medskip

Motivated by \cite{MPT} with adaptation to our setting,
we may define the spatial gradient on $\partial\mathcal O$ as above using the  \textit{Haj\l{}asz-Sobolev space}. Namely,
For a Borel function \(f:\partial\mathcal{O}\times\R\to\mathbb{R}\) we say that a Borel function \(g:\partial\mathcal{O}\times\R\to\mathbb{R}\) is a \textit{Haj\l{}asz upper spatial gradient} of \(f\) if
\[|f(X,t)-f(Y,t)|\leq |X-Y|(g(X,t)+g(Y,t))\qquad\textrm{for a.e. }(X,t),(Y,t)\in \partial\mathcal{O}\times\R.\]
We denote the collection of all Haj\l{}asz upper spatial gradients of \(f\) as \(\mathcal{D}_s(f)\). 
For a measurable function \( f \) on $\partial\Omega=\partial\mathcal O\times\R$ we define the Sobolev space \( \dot{L}_{1,1/2}^p(\partial\Omega) \) via the norm
\begin{equation}\label{xnorm}
\| f \|_{\dot{L}_{1,1/2}^p(\partial\Omega)} \coloneqq 
   \inf \{\|g\|_{L^p(\partial\Omega)}:\, g \in L^p(\partial\Omega)\cap\mathcal{D}_s(f) \} + \| D_{1/2}^t f \|_{L^p(\partial\Omega)}.
\end{equation}   
We can also see that 
\eqref{m22} still holds as it is a consequence of \eqref{m1}-\eqref{m1a} which can be shown for $f$ satisfying 
 \eqref{xnorm}. Also, it is easy to see that this defines the same space as previously for $\mathcal O$ Lipschitz (c.f. \cite{MPT}). 
\medskip

Now we are ready to define the Regularity problem.

\begin{definition}\label{def:regularity problem}
	Let \( 1 < p < \infty \), and the domain $\Omega$ be as above.
	We say that the \( L^p \) Regularity problem for \( L \) \( (R_L)_p \) is solvable 
		if for each continuous \( f \in \dot{L}_{1,1/2}^p(\partial\Omega) \) the solution to
	\[ \begin{cases}
		Lu = 0, \quad \text{in } \Omega,
        \\
		u = f, \quad \text{on } \partial\Omega,
	\end{cases} \]
	satisfies
    \begin{align}\label{eq:NSpaceBound}
        \| N[\nabla u] \|_{L^p(\partial\Omega)} \lesssim \| f \|_{\dot{L}_{1,1/2}^p(\partial\Omega)}.     
    \end{align}
\end{definition}

\begin{remark}
One might object that a condition of the type
\begin{align}\label{eq:NTimeBound}
    \| N[D_{1/2}^t u] \|_{L^p(\partial\Omega)} \lesssim \| f \|_{\dot{L}_{1,1/2}^p(\partial\Omega)},     
\end{align}
    should be satisfied too (c.f. \cite{Bro89,CRS,M, Nys06}). 
However, as it turns out, this condition is superfluous.
Indeed by \cite{D} we have 
\begin{align}\label{eq:TimeDerivative_not_needed}
    \| N[D_{1/2}^t u] \|_{L^p(\partial\Omega)} \lesssim  \| N[\nabla u] \|_{L^p(\partial\Omega)}+ \| f \|_{\dot{L}_{1,1/2}^p(\partial\Omega)},
\end{align}
    and clearly \eqref{eq:NTimeBound} follows from \eqref{eq:TimeDerivative_not_needed}
    and \eqref{eq:NSpaceBound}. This simplifies our task, as we only need to establish bounds for $N[\nabla u]$
    and the bounds for the half-time derivative will follow.
\end{remark}    
    
\subsection{Localization}    
    
We shall establish the following useful result that will be needed later.

\begin{proposition}\label{prop:Product_Rule}
Let $\Omega=\mathcal O\times\R$, where $\mathcal O\subset\R^n$ is a uniform domain with $n-1$-Ahlfors regular boundary. Assume that for some $p>1$
 $f\in \dot{L}^p_{1,1/2}(\partial\Omega)$ (with $ \dot{L}^p_{1,1/2}$ defined by \eqref{xnorm}).
 There exists $h\in L^p(\partial\Omega)$ with $\|h\|_{L^p(\partial\Omega)}\lesssim \|f\|_{\dot{L}^p_{1,1/2}(\partial\Omega)}$ for which the following holds:
 
Consider any parabolic ball $B\subset \R^n\times\R$ of radius $r$ centered at the boundary $\partial\Omega$, let $\Delta=B\cap\partial\Omega$ and  $\varphi$ be a $C_0^\infty(\mathbb R^n\times\R)$ cutoff function such that $\varphi=1$ on $2B$ and $\varphi=0$ outside $3B$ with $\|\partial_t\varphi\|_{L^\infty}\lesssim r^{-2}$ and $\|\nabla_x\varphi\|_{L^\infty}\lesssim r^{-1}$. 
 Then we have the following estimate:
$$\|(f-\textstyle\fint_{\Delta} f)\varphi\|_{\dot{L}^p_{1,1/2}(\partial\Omega)}\lesssim \|h\|_{L^p(32\Delta)}+\|D^{1/2}_t f\|_{L^p(32\Delta)}
\lesssim \|f\|_{\dot{L}^p_{1,1/2}(\partial\Omega)}.$$

Furthermore for all $(x,t)\in 8\Delta$ we have
\begin{equation}\label{245}
|D^{1/2}_t((f-\textstyle\fint_{\Delta} f)\varphi)(x,t)-D^{1/2}_tf(x,t)\varphi(x,t)|\lesssim
h(x,t).
\end{equation}
\end{proposition}

\begin{proof}  Recall that we know that we have for $f$ \eqref{m1ax} and hence
\begin{equation}\label{split1bb}
|f(x,t)-f(y,\tau)|\le d((x,t),(y,\tau))[f^\sharp(x,t)+f^\sharp(y,\tau)]
\end{equation}
for a.e. pair od points $(x,t),\,(y,\tau)\in\partial\Omega$.
We have to be very careful where we use the condition \eqref{split1bb} as this does not characterise the space of functions for which $\|D^{1/2}_tf\|_{L^p}<\infty$. 
We have lost a bit of regularity when we have used in the previous section the inclusion $\dot F^{1/2}_{p,2}\subset \dot F^{1/2}_{p,\infty}$ which is proper.  
However, we assume $\varphi\in C_0^\infty(\mathbb R^n\times\R)$ which is better than just $D^{1/2}_t\varphi\in L^p(\mathbb R^n\times\R)$. This is an extra piece of regularity that will make our argument work. 
We calculate $D^{1/2}_t((f-\textstyle\fint_{\Delta} f)\varphi)$. Clearly,
\begin{equation}
D^{1/2}_t((f-\textstyle\fint_{\Delta} f)\varphi)(x,t)=c\displaystyle\int_{\mathbb R}\frac{(f(x,t)-\textstyle\fint_{\Delta} f)\varphi(x,t)-(f(x,s)-\textstyle\fint_{\Delta} f)\varphi(x,s)}{|t-s|^{3/2}}ds
\end{equation}
\begin{equation}\label{split22}
=\varphi(x,t)D^{1/2}_t(f)(x,t)+c\int_{s\in\mathbb R\cap\{(x,s)\in 3\Delta\}}\frac{\varphi(x,t)-\varphi(x,s)}{|t-s|^{3/2}}(f(x,s)-\textstyle\fint_{\Delta} f)ds.
\end{equation}
The second term needs to be considered further. 
We need to think about when the term $\varphi(x,t)-\varphi(x,s)\ne 0$ as otherwise there is nothing to do. If the point $(x,t)$ is far from support of some enlargement of $\Delta$, say $8\Delta$, whenever $\varphi(x,t)-\varphi(x,s)\ne 0$ we will have $|t-s|\approx 2^ir^2$ for some $i=2,3,\dots,$ for all $s\in\mathbb R$ such that $(x,s)\in3\Delta$.
We use a crude bound for the numerator $|\varphi(x,t)-\varphi(x,s)|\le 1$. This allows us to conclude that when $(x,t)\notin 8\Delta$ then the second term will be bounded by
$$\left|\int_{\mathbb R}\frac{\varphi(x,t)-\varphi(x,s)}{|t-s|^{3/2}}(f(x,s)-\textstyle\fint_{\Delta} f)ds\right|
\lesssim (2^{-3i/2})r^{-1}\fint_{s\in3\Delta}|f(x,s)-\textstyle\fint_{\Delta} f|ds$$
\begin{equation}\label{split11}
\lesssim (2^{-3i/2})r^{-1} M^{t<4r^2}(|f-\textstyle\fint_{\Delta} f|)(x,t'),\quad\mbox{ for some } i\ge2.
\end{equation}
Here $(x,t')$ is an arbitrary point in $3\Delta$ and $M^{t<4r^2}$ is the Hardy-Littlewood maximal function in $t$-variable only on all balls of size $<4r^2$.

This leaves us to consider the case $(x,t)\in 8\Delta$. It is then possible that $|t-s|$ is arbitrary small for $s\in\mathbb R$ such that $\varphi(x,t)-\varphi(x,s)\ne 0$. We use the fact that $\varphi$ is Lipschitz in $t$-variable with Lipschitz constant $\approx r^{-2}$. We also split the integral to sets on which $|s-t|\approx 2^{-j}r^2$ for $j=0,1,2,\dots$. This gives us for the second term of \eqref{split22}:
\begin{equation}\label{splitzz}
\left|\int_{\mathbb R}\frac{\varphi(x,t)-\varphi(x,s)}{|t-s|^{3/2}}(f(x,s)-\textstyle\fint_{\Delta} f)ds\right|
\lesssim \sum_{j=0}^\infty{r^{-2}}\int_{|s-t|\approx 2^{-j}r^2}\frac{|f(x,s)-\textstyle\fint_{\Delta} f|}{|t-s|^{1/2}}ds.
\end{equation}
By introducing extra term $|t-s|^{1/2}$ into the denominator we again recognise a maximal function in $t$-variable. Thus continuing our calculation we have
$$\lesssim \sum_{j=0}^\infty 2^{-j/2}r^{-1} M^{t<16r^2}(|f-\textstyle\fint_{\Delta} f|)(x,t)\approx
r^{-1}M^{t<16r^2}(|f-\textstyle\fint_{\Delta} f|)(x,t).$$
We now estimate this maximal function. Let $(x,t)\in 8\Delta$. Using \eqref{sharp2} for $B=8\Delta$
we get that
\begin{equation}\label{sharp2x}
r^{-1}\left|f(x,t)-\fint_{\Delta}f\right|\lesssim f^\sharp(x,t)+\fint_{8\Delta}f^\sharp,
\end{equation}
and hence averaging over any time-only ball $\Delta'$ that contains $(x,t)$ we get that
$$
r^{-1}\left|\fint_{\Delta'}f(x,t)-\fint_{\Delta}f\right|\lesssim \fint_{\Delta'}f^\sharp(x,t)+\fint_{8\Delta}f^\sharp,
$$
and after taking sup over all such balls of size $t<16r^2$ we get that
$$r^{-1} M^{t<16r^2}(|f-\textstyle\fint_{\Delta} f|)(x,t)\lesssim  M^{t<16r^2}(f^\sharp)(x,t)
+M^{<8r}(f^\sharp)(x,t)=:h(x,t).$$
We take this as the definition of the function $h$. Clearly, since $M^t$ and $M$ are $L^p$ bounded for $p>1$ we have that $\|h\|_{L^p(\partial\Omega)}\lesssim  \|f\|_{\dot{L}^p_{1,1/2}(\partial\Omega)}$.

Combing this with \eqref{split22}  and \eqref{splitzz}
we obtain for $(x,t)\in 8\Delta$
\begin{align}
&&
|D^{1/2}_t((f-\textstyle\fint_{\Delta} f)\varphi)(x,t)-D^{1/2}_tf(x,t)\varphi(x,t)|\lesssim
h(x,t),
\end{align}
thus proving \eqref{245}.
After we integrate the inequality above over $8\Delta$ we obtain:
$$\int_{8\Delta} [D^{1/2}_t((f-\textstyle\fint_{\Delta}f)\varphi]^p\,dx\,dt\lesssim \displaystyle\int_{8\Delta} [|D^{1/2}_tf|^p+h^p]dx\,dt.$$

Next we estimate the same integral on $\mathbb R^n\setminus 8\Delta$. We only need to consider integrating over $(x,t)$ for which there is $s\in\mathbb R$ such that $(x,s)\in 8\Delta$. Let us call such $s=s_0$. Let us call $B$ the set of such points $x$. It follows that
$$\int_{\mathbb R^n\setminus 8\Delta} [D^{1/2}_t((f-\textstyle\fint_{\Delta}f)\varphi]^p\,dx\,dt =
\displaystyle\int_{(B\times \mathbb R)\setminus 8\Delta} [D^{1/2}_t((f-\textstyle\fint_{\Delta}f)\varphi]^p\,dx\,dt.$$
Recalling \eqref{split11} and using the estimate in \eqref{split22} we see that after further splitting the integral above into regions such that the distance of points $(x,t)$ to $4\Delta$ is approximately $2^{i}r^2$ we have that
$$\hspace{-7cm}\int_{\mathbb R^n\setminus 8\Delta} [D^{1/2}_t((f-\textstyle\fint_{\Delta}f)\varphi]^p\,dx\,dt$$
$$\lesssim\sum_{i=2}^\infty \displaystyle\int_{(B\times \{t\in\mathbb R:{ dist}((x,t),4\Delta)\approx 2^ir^2\}}
(2^{-3ip/2})r^{-p} [M^{t<4r^2}(|f-\textstyle\fint_{\Delta} f|)]^p(x,s_0)dx\,dt$$
$$\lesssim \sum_{i=2}^\infty 2^{i(1-3p/2)} \int_{32\Delta} h^p\,dx\,dt,$$
which shows that $D^{1/2}_t$ of our function belongs to $L^p$. For the spatial \lq\lq gradient" the argument is easier. With $h$ as above we have that
$$|(f-\textstyle\fint_{\Delta}f)\varphi(x,t)-(f-\textstyle\fint_{\Delta}f)\varphi(y,t)|\le
|x-y||f(x,t)-f(y,t)||\varphi(x,t)|$$
$$+|(f(y,t)-\textstyle\fint_{\Delta}f)||\varphi(x,t)-\varphi(y,t)|
\lesssim |x-y|[f^\sharp(x,t)+f^\sharp(y,t)+h(y,t)],
$$
where we have used the earlier estimate for $r^{-1} M^{t<16r^2}(|f-\textstyle\fint_{\Delta} f|)(y,t)$.
This is a good estimate if both $(x,t),(y,t)\in 8\Delta$.
If one of the points is not inside $8\Delta$ but $|x-y|\le r$ we trivially have
$$|(f-\textstyle\fint_{\Delta}f)\varphi(x,t)-(f-\textstyle\fint_{\Delta}f)\varphi(y,t)|=0.$$
and finally if $(x,t)\in 8\Delta$, $(y,t)\notin 8\Delta$ but $|x-y|>r$
 we have
 a trivial estimate
$$|(f-\textstyle\fint_{\Delta}f)\varphi(x,t)|\le |x-y|r^{-1}M_{t<16r^2}(|f-\textstyle\fint_{\Delta} f|)(x,t)
\lesssim |x-y|h(x,t),$$
which shows that for $G(x,t)=h(x,t)\chi_{8\Delta}$ it always holds that
$$|(f-\textstyle\fint_{\Delta}f)\varphi(x,t)-(f-\textstyle\fint_{\Delta}f)\varphi(y,t)|\lesssim
|x-y|[G(x,t)+G(y,x)].$$
Hence $G$ is the Hajłasz–Sobolev spatial gradient of the function $(f-\textstyle\fint_{\Delta}f)\varphi$ and furthermore $\|G\|_{L^p(\partial\Omega)}\lesssim \|h\|_{L^p(8\Delta)}
\lesssim \|f\|_{\dot{L}^p_{1,1/2}(\partial\Omega)}$ as claimed.
\end{proof}

\section{Parabolic results}\label{S3}

Below, we will list some known properties of solutions to \( Lu=0 \) in \( \Omega \).
We would like to highlight    \reflemma{lemma:Gradient L2toL1} and \reflemma{lemma:|u|bound} 
    that will be key in proving \reftheorem{thm:Reverse Holder}.  Results here that involve boundary have all been established in the setting we assume, that is $\Omega=\mathcal O\times \R$ and hence $\partial\Omega=\partial\mathcal O\times\R$ or on even more general boundaries (such as 1-sided NTA domains with time-symmetric Ahlfors regular boundary c.f. \cite{GH,GHMN}).

\begin{proposition}[Boundary Poincar\'e inequality \cite{DU,U}] \label{prop:Boundary Poincare}
    Let \( \Delta= \Delta_r(P_0) \) be a boundary ball of \(\partial\Omega\) 
        and suppose that \(u = 0\) on \( 5\Delta \).
    Then
    \begin{align}\label{eq:Poincare_inequality_two}
        \int_{T(\Delta)} \Big|\frac{u}{\delta}\Big|^2 
            \lesssim \int_{T(5\Delta)} |\nabla u|^2.
    \end{align}
\end{proposition}

\begin{proposition}[Parabolic Cacciopolli, Lemma 3.3 of Chapter I in \cite{HL01}]\label{prop:Cacciopolli}
For a weak solution \(u\) in \(Q(X,t,4r)\subset\Omega\) we have
\begin{align*}
    r^n(\max_{Q(X,t,r/2)}|u|)^2&\lesssim \sup_{s\in (t-r^2,t+r^2)}\int_{Q_{X}(X,r)\times \{s\}} |u|^2 dZd\tau
    \\
    &\lesssim \int_{Q(X, t,r)} |\nabla u|^2 dZd\tau
    \\
    &\lesssim \frac{1}{r^2}\int_{Q(X, t,2r)} |u|^2 dZd\tau.
\end{align*}
\end{proposition}

\begin{proposition}[Boundary Cacciopolli \cite{DU}]\label{prop:Boundary Cacciopolli}
    Let \(u\) be a solution in \(T(\Delta(P,s,2r))\) that vanishes in \(\Delta(P,s,2r)\). Then
    \[\int_{T(\Delta(P,s,r))}|\nabla u|^2\lesssim \frac{1}{r^2}\int_{T(\Delta(P,s,2r))}| u|^2.\]
\end{proposition}

\begin{proposition}\label{prop:Gradient_LptoL2} There exists $p>2$ such that if
     \( Q \) is a parabolic cube with \( 2Q \subset \Omega \) then
    \begin{align*}
        \Big( \fint_{Q} |\nabla u|^p \Big)^{1/p} 
        \lesssim \Big( \fint_{2Q} |\nabla u|^2 \Big)^{1/2}     
    \end{align*}
    Moreover, if \( u = 0 \) on \( 4\Delta \)
        we can replace \( Q \) with \( T(\Delta) \) in the above inequality.
\end{proposition}

From this we have the following:
\begin{lemma}\label{lemma:Gradient L2toL1}
    Let \( u = 0 \) on \( 4\Delta \) for some boundary ball $\Delta$.
    Then
    \begin{align}\label{eq:Gradient_L2toL1}
        \Big( \fint_{T(\Delta)} |\nabla u|^2 \Big)^{1/2} 
        \lesssim  \fint_{T(2\Delta)} |\nabla u|  .
    \end{align}
\end{lemma}
This follows from \refproposition{prop:Gradient_LptoL2} by a standard real-variable argument (c.f. 
\cite{She19}, for example). 

\begin{proposition}[Harnack inequality, Lemma 3.5 of Chapter I in \cite{HL01}]\label{prop:HarnackInequality}
    Let \((Z,\tau),(Y,s)\in Q(X,t,r)\) and \(u\) a non-negative weak solution on \(Q(X,t,2r)\). 
    Then for \(\tau<s\)
    \[u(Z,\tau)\leq u(Y,s)\exp\Big[c\Big(\frac{|Y-Z|^2}{|s-\tau|}+1\Big)\Big].\]
\end{proposition}

\begin{proposition}[Boundary Hölder continuity \cite{GH,DU}]\label{prop:BoundaryHölder}
    Let \(u\) be a solution on \(T(\Delta(x,t,2r))\) that vanishes on continuously on \(\Delta(x,t,2r)\). 
    Then there exists \(\alpha>0\) such that for \((Y,s)\in T(\Delta(x,t,r/2))\)
    \[u(Y,s)\lesssim c\left(\frac{\delta(Y,s)}{r}\right)^\alpha \max_{T(\Delta(x,t,r))}|u|.\]
    If additionally \(u\geq 0\) in \(T(\Delta(x,t,2r))\), then
    \begin{align}\label{eq:BoundaryHoelderNonnegativeCorkscrewEstimate}
        u(Y,s)\lesssim \left(\frac{\delta(Y,s)}{r}\right)^\alpha u(V^+(\Delta(x,t,2r))).
    \end{align}
\end{proposition}

By \cite{DU} we also have the existence of a Greens function for the operator \( L \) and its adjoint \( L^* \) with 
\[ G^*(X, t, Y, s) = G(Y, s,X, t). \]
    Moreover, let \( \omega(X,t,\Delta) \) denotes 
    the parabolic measure of \( \Delta \subset \partial\Omega \) with the pole at \( (X,t) \in \Omega \)
    (and \( \omega^* \) denotes the parabolic measure of $L^*$).
    We have the following comparison result
\begin{proposition}[\cite{GHMN}]\label{prop:Grean To Meas}
    Let \( (X,t), (Y,s) \in \Omega \) with \( s+8r^2 \leq t \), where \( r := \delta(Y, s) \). Let $Y^*$ be a point on $\partial\mathcal O$ such that $Y\in\tilde{\gamma}(Y^*)$.
Then 
\[ r^nG(X, t, Y, s) \sim \omega(X, t, \Delta(Y^*, s, r)) \]
Similarly, if \(  s-8r^2 \geq t \), then
\[ r^nG^*(X, t, Y, s) \sim \omega^*(X, t, \Delta(Y^*, s, r)) \]
\end{proposition}

As a consequence, for \( d\geq 10r \) we have
\begin{align}\label{eq:GreenToMeas}
    r^n G(V^-(\Delta_r),V^-(\Delta_d))  = r^n G^*(V^-(\Delta_d),V^-(\Delta_r)) \sim \omega^{*V^-(\Delta_d)}(\Delta_r). 
\end{align}

Furthermore, we have the following decay property for the Green's function:

\begin{lemma}\label{lemma:ParabolicGrwothEstimateForGwithalpha}
    There exists \(\alpha=\alpha(\lambda, n)\) such that for \((X,t),(Y,s)\in\Omega\) with \(t<s\)
    \begin{align}G(X,t,Y,s)\lesssim \delta(X,t)^\alpha \big(|X-Y|^2+(s-t)\big)^{-\frac{n+\alpha}{2}}.\label{eq:ParabolicGrwothEstimateForGwithalpha}\end{align}
\end{lemma}

The proof stems from private communications with  Dindo\v{s},  Li, and Pipher and can be found as Lemma 2.8.17 in Ulmer thesis (\cite{U}).

\begin{proposition}[Local comparison Principle \cite{DU}]\label{prop:Comparison_Principle}
    Let \( u \) and \( v \) be non-negative solutions to \( Lu = 0 \) in \( T(\Delta_{2r}(Y,s)) \) 
        with \( u=v=0 \) on \( \Delta_{2r}(Y,s) \)
    \begin{align*}
        \frac{u(X,t)}{v(X,t)} \approx \frac{u(V^+(\Delta_r(Y,s)))}{v(V^-(\Delta_r(Y,s)))}, 
        \quad (X,t) \in T(\Delta_{r}(Y,s)), 
    \end{align*} 
and $V^+$ ($V^-$) are time forward (backward)    corkscrew point relative to the surface ball.    
\end{proposition}

\begin{proposition}[Lemma 2.15 in \cite{KX}]\label{prop:Sup to L2 bound}
    For solutions to \( Lu = 0 \) and with \( u=0 \) on \( \Delta_{5r} \)
    \[ \sup_{T(\Delta_{2r})} |u| \lesssim \Big( \fint_{T(\Delta_{4r})} |u|^2 \Big)^{1/2}. \qed \]    
\end{proposition}

This implies the following:
\begin{lemma}\label{lemma:|u|bound}
    Let \( u \) be a weak solution such  that \( u = 0 \) on \( \Delta_{5r} \subset\Delta_{d}\), where $d=\mbox{diam}(\mathcal O)$.
    Then for \( (X,t) \in T(\Delta_{2r}) \)
    \[ |u(X,t)| \lesssim \frac{G(X,t,V^{-}(\Delta_d))}{G(V^{-}(\Delta_r),V^{-}(\Delta_d))} 
        \Big( \fint_{T(\Delta_{4r})} |u|^2 \Big)^{1/2}. \]
\end{lemma}

\begin{proof}
Let \( (X,t) \in T(\Delta_{2r}) \) and set \( u^+ \coloneqq \max(u,0)\; u^- \coloneqq \min(-u,0) \).
Next write \( u= u_1 - u_2 \), where \( u_1 \) and \( u_2 \) solve 
\begin{align*}
\begin{cases}
    Lu_1 = 0, \quad \text{in } T(2\Delta),
    \\
    u_1 = u^+,  \text{on } \partial_p T(2\Delta),
\end{cases}\qquad
\begin{cases}
    Lu_2 = 0, \quad \text{in } T(2\Delta),
    \\
    u_2 = u^-,  \text{on } \partial_p T(2\Delta),
\end{cases}
\end{align*}
    respectively. Here $\partial_pT(2\Delta)$ denotes the parabolic part of the boundary of $T(2\Delta)$.
Let \( j \in \{1,2\} \).
By \refproposition{prop:Comparison_Principle} 
\begin{align*}
    u_j(X,t) \lesssim \frac{G(X,t;V^{-}(\Delta_d))}{G(V^{-}(\Delta_r);V^{-}(\Delta_d))} 
        u_j(V^+(\Delta_r)).
\end{align*}
Next note that since \( V^+(\Delta_r) \in T(2\Delta) \) we can apply the maximum principle 
    (see e.g. \cite{Aro68}) 
    followed by \refproposition{prop:Sup to L2 bound} to see that
\[ u_j(V^+(\Delta_r)) \leq \sup_{\partial T(2\Delta)} |u| \lesssim \Big( \fint_{T(\Delta_{4r})} |u|^2 \Big)^{1/2}. \]
Thus
\[ |u(X,t)| \leq u_1(X,t) + u_2(X,t) \lesssim \Big( \fint_{T(\Delta_{4r})} |u|^2 \Big)^{1/2}, \]
    as desired.
\end{proof}

    
\section{Reverse Hölder inequality}\label{S:reverse Holder}
In this section we establish the Reverse Hölder inequality of Shen \cite{She06}
    in the parabolic case.

\begin{theorem}\label{thm:Reverse Holder}
    Let \( 1 \le p < \infty \) and suppose that \( (D_{L^*})_{p'} \) is solvable
        and suppose that \( u \) is a weak solution of $Lu=0$ with \( u=0 \) 
        on \(  \Tilde{\Delta} \eqqcolon 80 \Delta \) where \( \Delta = \Delta_r(P_0,\tau_0) \), with $r\lesssim \mbox{diam}(\mathcal O)$.
    Then
    \[ \Big( \fint_{\Delta} {N}[\nabla u]^p \Big)^{1/p}
        \lesssim \fint_{\Tilde\Delta} {N}[\nabla u], \]
as well as
\begin{equation}\label{zsaz}
     \Big( \fint_{\Delta} {N}[\nabla u]^p \Big)^{1/p}
        \lesssim \fint_{T(\Tilde\Delta)} |\nabla u|. 
 \end{equation}       
\end{theorem}

\textit{Proof:}
Let \( (X,t) \in \Gamma(P,\tau) \).
    If \( \delta(X,t) \geq r \), then 
\[ \tilde{S}(X,t) \coloneqq \{ (P,s) \in \partial\Omega : (X,t) \in \Gamma(P,s) \}
    = \tilde{S}(X) \times (t-\delta(X)^2,t+\delta(X)^2), \]
i.e.,  $\tilde{S}(X,t)$, $\tilde{S}(X)$ are backward in space-time (space) cones originating from  the point $(X,t)$
defined earlier in Section \ref{S:basics}.
It follows from the definition of $\tilde{S}$ that
 \( |\tilde{S}(X,t)| \geq  4|\tilde{S}(X)|r^2 \gtrsim r^{n+1}, \) and thus
\begin{align*}
    \Average^2[\nabla u](X,t) 
    \leq \inf_{\tilde{S}(X,t)} {N}[u]
    \leq \fint_{\tilde{S}(X,t)} {N}[u]
    \lesssim \fint_{80 \Delta} {N}[u].
\end{align*}
Thus we can disregard the points whose distance to the boundary is $r$ or larger and only consider the inequality for truncated non-tangential maximal function at the height $r$:
\[ \Big( \fint_{\Delta} {N}^r[\nabla u]^p \Big)^{1/p}
        \lesssim \fint_{\Tilde\Delta} {N}[\nabla u]. \]
By \refproposition{prop:Cacciopolli} the $L^2$ average of $\nabla u$ over $C_{\delta(X)/4}(X,t)$ can be controlled from above by the $L^2$ average of $|u/\delta|$ over $C_{\delta(X)/2}(X,t)$ which is further controlled by just the supremum of $|u/\delta|$ over $C_{\delta(X)/2}(X,t)$. It follows that

\[ {N}^r[\nabla u](P,\tau) 
    \lesssim {N}_{\sup}^r[u/\delta](P,\tau) , \]
where 
$${N}_{\sup}^r[v](P,\tau)=\sup_{\Gamma^{2r}_{\tilde{a}}(P,\tau)}|v|,$$
for some sufficiently large $\tilde{a}>1$.

Let \( d \geq 10r \). 
By \reflemma{lemma:|u|bound}
\[ N_{\sup}^r[u/\delta](P,\tau) 
    \lesssim \frac{1}{G(V^{-}(\Delta),V^-(\Delta_{d}))} \Big( \fint_{T(\Delta_{4r})} |u|^2 \Big)^{1/2} 
    N^r_{\sup}\left[\frac{G(\cdot,V^-(\Delta_{d}))}{\delta(\cdot)}\right], \]
    and threfore
\[ \Big( \fint_{\Delta} {N}^r[\nabla u]^p \Big)^{1/p}
    \lesssim  \Big( \fint_{T(\Delta_{4r})} |u|^2 \Big)^{1/2} \frac{1}{G(V^{-}(\Delta),V^-(\Delta_{d}))}
        \Big( \fint_{\Delta} N^r_{\sup}\left[G(\cdot,V^-(\Delta_{d}))/\delta(\cdot)\right]^p \Big)^{1/p}. \]
Let \( \omega^* = \omega^{*V^-(\Delta_{d})} \)  be the parabolic measure of the adjoint operator. 
Observe that by \eqref{eq:GreenToMeas} 
\[ \frac{1}{G(V^{-}(\Delta),V^-(\Delta_{d}))} \sim \frac{r^{n}}{\omega^*(\Delta))}, \]
It follows the for \( \rho := \delta(X,t) \) we have
\[ \frac{G(X,t,V^-(\Delta_{d}))}{\delta(X,t)} 
    \sim \frac{\omega^*(\Delta_{\rho}(P,\tau + 100\rho))}{\rho^{n+1}}
    \lesssim_n \frac{\omega^*(\Delta_{11\rho}(P,\tau))}{\rho^{n+1}}. \]
Thus 
\[ N^r_{\sup}[G(\cdot,V^-(\Delta_{d}))/\delta(\cdot)] \leq M^r\left[\frac{d\omega^*}{d\sigma}\right](P,\tau),\]
    where as before $M^r$ denotes the usual truncated maximal function which is bounded on $L^p$ for $p>1$.
    It follows that
\[ \Big( \fint_{\Delta} N_{\sup}^r[G(\cdot,V^-(\Delta_{d}))/\delta(\cdot)]^p \Big)^{1/p} 
    \lesssim \Big( \fint_{4\Delta} \Big| \frac{d\omega^*}{d\sigma} \Big|^p\Big)^{1/p}
    \lesssim r^{-n-1} \omega^*(\Delta), \]
    where in the final inequality we have used the fact that \( \omega^* \in B_p(d\sigma) \),
    a fact that follows from the solvability of \( (D^*)_{p'} \) and  that the measure $\omega^*$ is doubling.
Hence
\begin{align*}
    \Big( \fint_{\Delta} {N}^r[\nabla u]^p \Big)^{1/p}
    &\lesssim \frac{r^{n}}{\omega^*(\Delta)} r^{-n-1} \omega^*(\Delta) 
        \Big( \fint_{T(\Delta_{4r})} |u|^2 \Big)^{1/2}
    =  \Big( \fint_{T(\Delta_{4r})} |u/r|^2 \Big)^{1/2}
    \\
    &\lesssim \Big( \fint_{T(\Delta_{4r})} |u/\delta|^2 \Big)^{1/2}
    \lesssim \Big( \fint_{T(\Delta_{20r})} |\nabla u|^2 \Big)^{1/2}
    \lesssim  \fint_{T(\Delta_{40r})} |\nabla u| ,
\end{align*}
    where in the last two inequalities we have used \refproposition{prop:Boundary Poincare}
    and \refproposition{lemma:Gradient L2toL1}. This gives us \eqref{zsaz}. 
It remains to show that
\begin{equation}\label{zmac}
    \int_{T(\Delta_{40r})} |\nabla u| \lesssim  r\int_{\Delta_{40r}} N[\nabla u], 
\end{equation}
which is a rather straightforward calculation using Fubini's theorem and Whitney decomposition of the set 
$T(\Delta_{40r})$. We omit the details.\qed\medskip

\section{Proof of \reftheorem{thm:Dp_implies_Rq}}\label{S:duality}

\subsection{Unbounded case}
For this section we assume that \( \diam(\Base) = \infty \). What this means is that for any boundary point
$P\in\partial\mathcal O$ there exists interior corkscrew balls of arbitrary large size as the domain is unbounded.

\subsubsection{Almost good lambda argument}
The bulk of the proof of \reftheorem{thm:Dp_implies_Rq} is proving the \lq\lq almost \( \lambda \)-inequality" \reflemma{lemma:Almost lambda inequality}.
Let \( p > q \), and set
\[ E(\lambda) \coloneqq \{ M[N[\nabla u]^q] > \lambda \}
    \WORD{and} A = A(\EPS) \coloneqq 1/(2\EPS)^{q/p}. \]
Let \( g \in L^p(\partial\Omega) \) be a Haj\l{}asz upper spatial gradient for \( f \)
    and let \( h \) be like in \refproposition{prop:Product_Rule}.
We have the following inequality:
\begin{lemma}\label{lemma:Almost lambda inequality}
    Let \( 1<q<p<\infty \).
    Suppose that \( (R_L)_q \) and \( (D_{L^*})_{p'} \) are solvable.
    There exists \( \EPS,\gamma \) depending only on \( p,q,n,L,\Omega \) such that
        for all \( \lambda > 0 \)
    \begin{align}\label{eq:Lambda_inequality}
        |E(A\lambda)| \leq \EPS |E(\lambda)| 
            + |\{ M[|g|^q + |D_{1/2}^t f|^q + |h|^q] > \gamma\lambda \}|,
    \end{align}
\end{lemma}
\textit{Proof:}  
Let \( \EPS \in (0,1) \) be a small constant to be determined.
By \cite{Christ} there exists a family of \lq\lq dyadic cubes" such that:
\begin{enumerate}
    \item If \(l\geq k\) then either \(Q_\beta^l\subset Q_\alpha^k\) or \(Q_\beta^l\cap Q_\alpha^k=\emptyset\),
    \item For each \((k,\alpha)\) and \(l<k\) there is a unique \(\beta\) so that \(Q_\alpha^k\subset Q_\beta^l\),
        \item Each \(Q_\alpha^k\) contains a ball \( \Delta_\alpha^k \), 
        with \( r(\Delta) = 2^{-k} \)
    \item There exists a constant \(C_0>0\) 
        such that \(2^{-k}\leq \mathrm{diam}(Q_\alpha^k)\leq C_0 2^{-k}\),
\end{enumerate}
    where the diameter is taken in the parabolic metric.
The last listed property implies together with the Ahlfors regularity of the surface measure that \(\sigma(Q_\alpha^k)\approx 2^{-k(n+1)}\). 
\\

Let \( E(\lambda) = \bigcup_k Q_k \) 
    be a Whitney decompotition of \( E(\lambda) \)
    into a disjoint collection of maximal such \lq\lq dyadic subcubes". 
In this place we use unboundedness of our domain, 
    as on a bounded domain for diam$(Q_k)>>\mbox{diam}(\mathcal O)$ 
    the cubes become \lq\lq misshaped" (with long side in $t$ much larger that the size in spatial variables). 
We claim that we can choose \( \EPS,\gamma,C_0 \) 
    such that if \( Q_k \) is a cube with the property 
\begin{align}
    Q_k \cap \{ M[|D_{1/2}^t f|^q + |g|^q + |h|^q] \leq \gamma\lambda \} \neq \EMP,
\end{align}
    then
\begin{align}
    |E(A\lambda) \cap Q_k| \leq \EPS |Q_k|.
\end{align}
The desired estimate \eqref{eq:Lambda_inequality} then follows by summation in $k$.
Let us now fix a cube \( Q_k \) and drop the index, 
    writing \( Q = Q_k, \Delta = \Delta_k \) and letting \( r = r(\Delta) \).
To prove the desired implication, first observe that
\begin{align}\label{eq:Maximal function}
    M[|{N}[\nabla u]|^q](P,\tau) 
    \leq \max\{ M^{C_0 r}[{N}[\nabla u]^q](P,\tau), C_1 \lambda \}
\end{align}
    for any \( (P,\tau) \in Q \) and with \( C_1 = C_1(\mathcal O) \).
This is because \( Q \) is maximal and so \( 2C_0 \Delta \not\subset E(\lambda) \).
Assume that \( A \geq C_1 \).
It follows that
\begin{align}\label{eq:localize_to_Q}
    |Q \cap E(A\lambda)| \leq |\{ (P,\tau) \in Q : 
        M^{C_0r}[{N}[\nabla u]^q](P,\tau) > A \lambda \}|.
\end{align}

Next let \( \varphi = \varphi_Q \) be a smooth cut-off function on \( \R^{n+1} \) such that 
\[ \varphi = 1 \text{ on } C_2Q, \quad \varphi= 0, \text{ on } \partial\Omega \setminus 2C_2Q,
   \quad |\nabla \varphi| \lesssim \ell(Q)^{-1}, \]
   where \( C_2 \geq 160C_0 \).
Define 
\[ f_Q \coloneqq f - \alpha_Q, \quad \alpha_Q \coloneqq \fint_{80Q} f \]
Now let \( v = v_Q \) be the unique weak solution of 
\begin{align*}
    \begin{cases}
        Lv = 0, &\text{ in } \Omega,\\
        v = \varphi f_Q, &\text{ on } \partial\Omega,
    \end{cases} 
\end{align*}
Let \( \bar{p} > p \).
In view of \eqref{eq:localize_to_Q}, we have
\begin{align*}
    |Q \cap E(A\lambda)| 
    & \leq |\{ P \in Q : M^{C_0r}[{N}[\nabla(u-v)]^q](P,\tau) > \frac{A\lambda}{2^q} \}|
    \\
        &\quad+ |\{ P \in Q : M^{C_0r}[{N}[\nabla v]^q](P,\tau) > \frac{A\lambda}{2^q} \}|
    \\
    & \lesssim (A\lambda)^{-\Bar{p}/q} \int_{2C_0 \Delta} {N}[\nabla (u-v)]^{\bar{p}} 
        + (A\lambda)^{-1} \int_{2C_0 \Delta} {N}[\nabla v]^{q} \eqqcolon J_1 + J_2, 
\end{align*}
    where we have used the weak \( (\frac{\Bar{p}}{q},\frac{\Bar{p}}{q}) \) and \( (1,1) \)
    estimates for the maximal function.
\\

We start with the term \( J_2 \).
Since \( (R_L)_q \) is solvable, and applying \refproposition{prop:Product_Rule}, we have that for some large $C_3 >>1$ we have that
\begin{align*} 
    \| {N}[\nabla v] \|_{L^q(\partial\Omega)}
    &\lesssim \| \varphi f_Q \|_{\Dot{L}_{1,1/2}^q(\partial\Omega)} 
    \lesssim \| f \|_{\Dot{L}_{1,1/2}^q(C_3 \Delta)} + \| h \|_{L^q(C_3 \Delta)}
    \\
    &\lesssim \int_{C_3 \Delta}  |D_t^{1/2} f|^q + |g|^{q} + |h|^q
    \lesssim \gamma\lambda|Q|\
\end{align*}
    and hence 
\begin{align}\label{eq:J_2 bound}
    J_2 \lesssim A^{-1}\gamma|Q| \lesssim \gamma \EPS^{q/p}|Q|.     
\end{align}

Now we deal with the first term \( J_1 \).
Observe that \( (u-\alpha) - v \) is a weak solution 
    whose boundary data \( f_Q(1-\varphi) \) vanishes on \( C_2\Delta \).
Also note that the solvability of \( (D_{L^*})_{p'} \) implies that \( (D_{L^*})_{\Bar{p}'} \) 
    is solvable for some \( \Bar{p} > p \) due to properties of the $B_p$ weights.
It then follows by \reftheorem{thm:Reverse Holder} that
\begin{align*}
    \int_{2C_0\Delta} {N}[\nabla (u-v)]^{\bar{p}} 
    &= |2C_0\Delta| \fint_{2Q} {N}[\nabla (u-v)]^{\bar{p}} 
    \lesssim |Q| \Big( \fint_{C_2 \Delta} {N}[\nabla (u-v)] \Big)^{\Bar{p}}
    \\
    &\lesssim |Q| \Bigg[ \Big( \fint_{C_2\Delta} {N}[\nabla u]^q \Big)^{\Bar{p}/q}
        + \Big( \fint_{C_2\Delta} {N}[\nabla v]^q \Big)^{\Bar{p}/q} \Bigg].
\end{align*}
Now, we already know that \( 
\left(\fint_{C_2\Delta}\Tilde{N}[\nabla v]^q\right)^{1/q} \lesssim \lambda\gamma \)
    and hence it remains to consider the first term.
Note that since \( C_2\Delta \not\subset E(\lambda) \) 
    there exists a point \( (P',\tau') \in C_2\Delta \) such that \( M[N[\nabla u]^q](P',\tau') \leq \lambda \).
Let \( \Delta' \) be a ball obtained by translating \( \Delta \) to be centered at \( (P',\tau') \),
    then
\[ \fint_{C_2\Delta} N[\nabla u]^q \lesssim \fint_{2C_2 \Delta'} N[\nabla u]^q \leq 
    M[N[\nabla u]^q](P',\tau') \leq \lambda. \]
Thus 
\begin{align}
    J_1 \lesssim \frac{1}{(A\lambda)^{\Bar{p}/q}} \lambda^{\Bar{p}/q} |Q| 
    = A^{-\Bar{p}/q} |Q| = \EPS^{\Bar{p}/p} |Q| = \EPS|Q| \EPS^{\Bar{p}/p-1}     
\end{align}

 and combining this with \eqref{eq:J_2 bound} we get
\[ |Q \cap E(A\lambda)| \lesssim J_1 + J_2 \lesssim \EPS|Q|(\EPS^{\Bar{p}/p-1} + \gamma \EPS^{-q/p-1}). \]
Hence, first choosing \( \EPS \) small enough that \( \EPS^{\Bar{p}/p-1} \leq 1/2 \)
    and then \( \gamma \) small enough that \( \gamma \EPS^{q/p-1} \leq 1/2 \) we indeed have that
\[ |Q \cap E(A\lambda)| \leq \EPS|Q|, \]
    thus completing our the proof.\qed\\

We are now ready to finish the proof of  \reftheorem{thm:Dp_implies_Rq}$\,$ in the unbounded case, i.e. when \( \diam(\Base) = \infty \).
Like in \cite{She06} we leverage \reflemma{lemma:Almost lambda inequality}
    for a change of variables to see that
$$ \int_{\partial\Omega} {N}[\nabla u]^p 
    \leq \int_{\partial\Omega} M[N[\nabla u]^q]^{p/q}
        \sim \lim_{K\to\infty }\int_0^K |E(\lambda)|\lambda^{p/q-1}d\lambda.$$    
For a fixed$K$  the integral on the righthand side is finite as   $\lambda^{p/q-1}\le K^{p/q-1}<\infty$.
Thus for a fixed $K$ we have    
    
\begin{align*}
   \int_0^K |E(\lambda)&|\lambda^{p/q-1}d\lambda 
    \leq  \frac{1}{2\EPS} \int_0^{K/A} |E(A\lambda)|\lambda^{p/q-1}d\lambda
    \\    
    &\leq \frac{1}{2} \int_0^{K/A} |E(\lambda)|\lambda^{p/q-1}d\lambda
            + \frac{1}{2\EPS}\int_0^{K/A} |\{ M[|D_{1/2}^t f|^q + |g|^q + |h|^q] > \gamma\lambda \}|\lambda^{p/q-1}d\lambda
    \\
    &\lesssim_{\EPS} \int_0^{K/A} |\{ M[|D_{1/2}^t f|^q + |g|^q + |h|^q] > \gamma\lambda \}|\lambda^{p/q-1}d\lambda,
\end{align*}
where in the last step since $K/A\le K$ we have  moved the term  $\frac{1}{2} \int_0^{K/A} |E(\lambda)|\lambda^{p/q-1}d\lambda$  to the lefthand side. We take limit $K\to\infty$. Recall, that
 \( h \) is as in \refproposition{prop:Product_Rule}
    and \( g \in L^p(\partial\Omega) \) is a Haj\l{}asz upper spatial gradient for \( f \). 
Thus \( \| h \|_{L^p(\partial\Omega)} \lesssim \| D_t^{1/2} f \|_{L^p(\partial\Omega)} \)
    and \( \| g \|_{L^p(\partial\Omega)} \lesssim \| f \|_{\Dot{L}_{1}^p(\partial\Omega)} \)
    therefore we have that
\begin{align*}
  \int_{\partial\Omega} {N}[\nabla u]^p\lesssim  \int_0^\infty &|\{ M[|D_{1/2}^t f|^q + |g|^q + |h|^q] > \gamma\lambda \}|\lambda^{p/q-1}d\lambda
    \\
    &\lesssim_{\gamma} \int_{\partial\Omega} M[|D_{1/2}^t f|^q + |g|^q + |h|^q]^{p/q}
    \\
    &\lesssim \int_{\partial\Omega} |D_{1/2}^t f|^p + |g|^p + |h|^p
    \lesssim \| f \|_{\Dot{L}_{1,1/2}^p(\partial\Omega)}^p,
\end{align*}
    completing the proof.\qed\\

\subsection{Bounded case}

Now suppose that \( \diam(\mathcal{O}) < \infty \),
    by rescaling we may without loss of generality assume that \( \diam(\mathcal{O}) = 1 \).
The first problem we run into is that the reverse Holder inequality from
    \reftheorem{thm:Reverse Holder} does not work for boundary cubes \( \Delta \) 
    with \( \diam(\Delta) >> 1 \).
For this reason we can prove an analogue of \reflemma{lemma:Almost lambda inequality} but only for local parabolic cubes of a finite size $\sim 1$.  Subsequently, we can only obtain local bounds for the $L^p$ norm of  $N[\nabla u]$ restricted to boundary parabolic cubes of size $\sim 1$.
To overcome this we will split our domain into subdomains of finite diameter (in space and time). \medskip

Denote by  \( \Omega_i \coloneqq \Base \times (i,i+1) \)
    and let \( \Delta_i = \partial_p \Omega_i = \partial\Base \times (i,i+1) \) for $i\in\Z$. There exists a nonnegative $C^1$ function $\psi\in C_c^1(\R)$, supp $\psi\subset [-1,2]$ such that  
$$\psi_i(\cdot)=\psi(\cdot+i), \qquad |\psi'|\lesssim 1,\quad\int_{\R}\psi=1,\qquad \sum_{i\in \Z} \psi_i \equiv 1.$$
Clearly,  \( \| D_{1/2}^t \psi \|_{L^r(\R)} \lesssim 1 \),  for each \( 1<r<\infty \).
Set
\[ f_{\Delta_i} \coloneqq \fint_{{\Delta}_{i}} f. \]
    Finally, let  \( \Psi_k \coloneqq \sum_{i \geq k} \psi_i \) and note that \( \psi_i = \Psi_{i+1} - \Psi_{i} \) and 
    $\Psi_k\equiv 1$ for on $[i+2,\infty)$.
    
Let \( w_i \) solve the PDE
\[ \begin{cases}
    Lw_i = 0, &\text{in } \Omega, 
    \\
    w_i = (f-f_{\Delta_i})\psi_i &\text{on } \partial\Omega,
\end{cases} \]
    If \( w \coloneqq \sum_{i \in \Z} w_i \) then  \( v \coloneqq u-w \) solves \( Lv = 0 \)
    with boundary data 
\[ \sum_{i \in \Z} f_{\Delta_i} \psi_i = \sum_{i \in \Z} f_{\Delta_i} (\Psi_{i+1} - \Psi_{i})
    = \sum_{i \in \Z} (f_{\Delta_{i-1}} -f_{\Delta_i}) \Psi_{i}. \]
Thus we can write \( v = \sum_{i \in \Z} v_i \) as where \( v_i \) solves
\[ \begin{cases}
    Lv_i = 0, &\text{in } \Omega, 
    \\
    v_i = (f_{\Delta_{i-1}} -f_{\Delta_i})\Psi_i &\text{on } \partial\Omega.
\end{cases} \]

If follows that we have decomposed the solution $u$ to $Lu=0$ with $u\big|_{\partial\Omega}=f$ into
$$u=\sum_i v_i+\sum_i w_i,$$
such that $w_i\big|_{\partial\Omega}$ is supported only on $\partial\Base \times (i-1,i+2)$ and $v_i\big|_{\partial\Omega}=0$ for $t<i-1$ and $v_i\big|_{\partial\Omega}$ is constant for $t>i+2$.\vglue1mm

Let us also define $\tilde{v}_i$, such that 
\( \tilde{v}_i \) solves
\[ \begin{cases}
    L\tilde{v}_i = 0, &\text{in } \Omega, 
    \\
    \tilde{v}_i = (f_{\Delta_{i-1}} -f_{\Delta_i})\Psi_i (1-\Psi_{i+20})&\text{on } \partial\Omega.
\end{cases} \]
Note that by considering the solution $\tilde{v_i}$ instead of
$v_i$, since $v_i\big|_{\partial\Omega}(P,\tau)=\tilde{v}_i\big|_{\partial\Omega}(P,\tau)$ for $\tau\le i+15$ 
    we have $N[\nabla v_i](P,\tau)=N[\nabla \tilde{v}_i](P,\tau)$ for all $\tau\le i+10$.
However, $\tilde{v_i}$ is more similar to $w_i$ as both are now compactly supported and half derivative of 
$\Psi_i (1-\Psi_{i+20})$ has faster decay than half derivative of $\Psi_i$.
\vglue1mm

Let us study the properties of $N[\nabla v_i]$ and $N[\nabla w_i]$. Without loss of generality, let
$v$ be either $v_i$ or $w_i$ for some $i\in\Z$. It follows that $v\equiv 0$ for $t\le i-1$. Thus
$N[\nabla v](P,\tau)\equiv 0$ for all $\tau\le i-5$.

When $i-5<\tau<i+10$ we may study $N[\nabla v](P,\tau)$ by considering the $(R_L)_p$ boundary value problem. As we have already observed above, \reflemma{lemma:Almost lambda inequality} only holds on scales comparable to the diameter of $\mathcal O$. Thus by modifying the calculation for the unbounded case we can obtain an estimate on a bounded (in time) portion of $\partial\Omega$, namely that

$$\int_{\cup_{k=i-5}^{i+9}\Delta_k} {N}[\nabla v]^p \lesssim \|v\|^p_{\dot{L}^p_{1,1/2}(\partial\Omega)}.$$ 
This calculation differs only slightly to the one we gave in the unbounded case, and is actually more similar to what Shen did in \cite{She06} as he only considered the case of a bounded domain.
Hence we obtain for $w_i$ and $v_i$ (here using $\tilde{v_i}$ as on this portion of the domain they coincide):

\begin{equation}\label{eq59}
\int_{\cup_{k=i-5}^{i+9}\Delta_k} {N}[\nabla w_i]^p \lesssim \|(f-f_{\Delta_i})\psi_i)\|^p_{\dot{L}^p_{1,1/2}(\partial\Omega)},\quad \int_{\cup_{k=i-5}^{i+9}\Delta_k} {N}[\nabla v_i]^p \lesssim |f_{\Delta_{i-1}} -f_{\Delta_i}|^p.
\end{equation}
Here, the second term is already simplified, as we use quietly the fact that $\Psi_i (1-\Psi_{i+20})$ has bounded half time-derivative on $L^p(\partial\Omega)$ for all $p>1$.\vglue1mm

Finally, for $\tau\ge i+10$ we will use the fact that $\nabla v$ decays exponentially. 

\begin{lemma}\label{lemma:Exponential Decay} Let $p\ge 1$ and assume that the Dirichlet problem $(D)_{p'}^*$ 
is solvable. (When $p=1$ we assume that $(D)_{q}^*$ is solvable for some $q>1$).
Let $v$ be a solution to $Lv=0$ and
    suppose that \( v = const. \) on \( \partial\Base \times (i+2,\infty) \).
    There exists an \( \alpha>0 \) such that for all $j>i+9$ we have that
    \begin{align*}
        \| {N}[\nabla v] \|_{L^p(\Delta_j)}
        \lesssim e^{-\alpha|i-j|} \| {N}[\nabla v] \|_{L^p(\Delta_{i+3})}.     
    \end{align*}
\end{lemma}

Hence by \eqref{eq59} we have for all $j>i+9$:
\begin{equation}\label{eq60}
\int_{\Delta_j} {N}[\nabla w_i]^p \lesssim e^{-\alpha p|i-j|}\|(f-f_{\Delta_i})\psi_i)\|^p_{\dot{L}^p_{1,1/2}(\partial\Omega)},\quad \int_{\Delta_j} {N}[\nabla v_i]^p \lesssim e^{-\alpha p|i-j|}|f_{\Delta_{i-1}}-f_{\Delta_i}|^p.
\end{equation}

In order to establish \reflemma{lemma:Exponential Decay} we start by proving the $L^2$ decay. Assume first that the coefficients of $A$ as smooth. 

\begin{proposition}\label{prop:Smooth Exponential Decay}
    Let \( u \) be such that \( {L}u = 0 \) on \( \partial \Base \times [0,\infty) \) and also 
    Moreover suppose that \( u=0 \) on \( \partial \Base \times [0,\infty) \). 
    Then there exists an \( \beta >0 \) such that for \( t>0 \)
    \begin{align}
        \| u(\cdot,t) \|_{L^2(\mathcal{O})} 
        \lesssim e^{-\beta t} \| u(\cdot,0) \|_{L^2(\mathcal{O})} \label{eq:Exponential Decay}.
    \end{align}
\end{proposition}

\begin{proof}
    Set \( \eta(t) \coloneqq \| u(\cdot,t) \|_{L^2(\Base)} \).
    For \( t > 0 \)
    \begin{align*}
        2\eta(t)\eta'(t) &= \partial_t (\eta(t)^2) = 2\int_{\Omega_{t}} u\partial_t u
        = - 2 \int_{\Omega_{t}} A \nabla u \cdot \nabla u
        \leq -2\lambda \int_{\Omega_t} |\nabla u|^2, 
    \end{align*}
    where $\Omega_t=\mathcal O\times\{t\}$.
    Since \( u=0 \) on \( \partial\Omega_t \), applying the Sobolev inequality on $\Omega_t$ 
        we have that 
    \begin{align*}
        -\int_{\Omega_t} |\nabla u|^2 
        \lesssim - \int_{\Omega_t} |u|^2
        = -\eta(t)^2.
    \end{align*}
    Hence \( \eta'(t) \lesssim -\eta(t) \)
        and we may apply Gronwall's inequality to deduce \eqref{eq:Exponential Decay}.
\end{proof}

Observe that this calculation holds on very general domains $\mathcal O$ as it only requires the Sobolev inequality. The assumption that $A$ is smooth can be removed by approximation as the estimate we are seeking only depends on the ellipticity constant of the operator $A$.\vglue1mm

Assume now that $\| {N}[\nabla v] \|_{L^p(\Delta_{i+3})}<\infty$ and that $v\big|_{\partial\Omega}\equiv0$ for all times $\ge i+2$.
Clearly, then also 
$$\| {N}[\nabla v] \|_{L^1(\Delta_{i+3})}\lesssim \| {N}[\nabla v] \|_{L^p(\Delta_{i+3})}<\infty,$$
and thus by \eqref{zmac}

$$\int_{\Omega_{i+3}}|\nabla v|\lesssim \| {N}[\nabla v] \|_{L^1(\Delta_{i+3})}\lesssim \| {N}[\nabla v] \|_{L^p(\Delta_{i+3})}.$$
By applying \reflemma{lemma:Gradient L2toL1} as we satisfy 
$v\big|_{\partial\Omega}(P,\tau)=0$ for $\tau>i+2$ we obtain that for some $\tau\in (i+3,i+4)$
$$\left(\int_{\mathcal O\times\{\tau\}}|\nabla v|^2\right)^{1/2}\lesssim \| {N}[\nabla v] \|_{L^p(\Delta_{i+3})}<\infty,$$
and finally by Sobolev (given that $v$ vanishes on the lateral boundary of $\Omega_{i+3}$): 
$$\left(\int_{\mathcal O\times\{\tau\}}|v|^2\right)^{1/2}\lesssim \| {N}[\nabla v] \|_{L^p(\Delta_{i+3})}<\infty.$$

Fix now some $j>i+9$. We propagate the $L^2$ initial data given on $\mathcal O\times\{\tau\}$ using  \refproposition{prop:Smooth Exponential Decay} to $t\in (j+7,j+12)$. It follows that for all such $t$ we have 

$$\| v(\cdot,t) \|_{L^2(\mathcal{O})} \lesssim e^{-\beta(j-i-3)} \| v(\cdot,\tau) \|_{L^2(\mathcal{O})}
 \lesssim e^{-\beta |i-j|} \| {N}[\nabla v] \|_{L^p(\Delta_{i+3})}<\infty.$$
By integrating over the interval $(j+7,j+12)$ and then applying \refproposition{prop:Sup to L2 bound} 
\begin{equation}\label{zmez}
\sup_{\Omega_{j+8}\cup\Omega_{j+9}\cup\Omega_{j+10}}|v|\lesssim e^{-\beta |i-j|} \| {N}[\nabla v] \|_{L^p(\Delta_{i+3})}<\infty.
\end{equation}

Finally, we repeat the argument we gave in the proof of \reftheorem{thm:Reverse Holder} where we estimated 
${N}^r[\nabla v](P,\tau)$ by ${N}_{\sup}^r[v/\delta](P,\tau)$ and then using the Green's function eventually we have obtained (for $r=1$ as the diameter of the domain is $1$):
\begin{align*}
    \Big( \fint_{\Delta_{j+9}} {N}[\nabla v]^p \Big)^{1/p}
    &\lesssim  \fint_{\Omega\cap (j+8.5,j+10.5)} |\nabla v|.
\end{align*}
The final step is to estimate the above righthand side by the lefthand side of \eqref{zmez} which can be done thanks to \refproposition{prop:Boundary Cacciopolli} (boundary Cacciopolli) which finally gives us that
\begin{align*}
    \Big( \int_{\Delta_{j+9}} {N}[\nabla v]^p \Big)^{1/p}
    &\lesssim  e^{-\beta |i-j|} \| {N}[\nabla v] \|_{L^p(\Delta_{i+3})}<\infty.
\end{align*}
Form this the claim of \reflemma{lemma:Exponential Decay} follows by applying the above estimate to 
the function $v-const.$\qed\medskip

We are ready to add up the estimates to obtain a bound for the $L^p$ norm of $u$. Fix $i\in \Z$
and consider $N[\nabla u]$ on $\Delta_i$. Since $N$ is sub-additive, clearly
$$N[\nabla u](P,\tau)\le \sum_j \left[N[\nabla v_j]+N[\nabla w_i]\right](P,\tau),$$ 
and therefore 
$$\|N[\nabla u]\|_{L^p(\Delta_i)}\le \sum_j \left[\|N[\nabla v_j]\|_{L^p(\Delta_i)}+\|N[\nabla w_j]\|_{L^p(\Delta_i)}\right]$$ 
$$\le \sum_{j=i-9}^{i+5}\left[\|N[\nabla v_j]\|_{L^p(\Delta_i)}+\|N[\nabla w_j]\|_{L^p(\Delta_i)}\right]+\sum_{j<i-9}\left[\|N[\nabla v_j]\|_{L^p(\Delta_i)}+\|N[\nabla w_j]\|_{L^p(\Delta_i)}\right]$$
$$\hspace{-7.5cm}\lesssim \sum_{j=i-9}^{i+5}\left[\|(f-f_{\Delta_j})\psi_j)\|_{\dot{L}^p_{1,1/2}}+|f_{\Delta_{j-1}}-f_{\Delta_j}|\right]$$
$$+\sum_{j<i-9}e^{-\alpha|i-j|}\left[\|(f-f_{\Delta_j})\psi_j)\|_{\dot{L}^p_{1,1/2}}+|f_{\Delta_{j-1}}-f_{\Delta_j}|\right],
$$
using \eqref{eq59} and \eqref{eq60}. We raise both sides to the $p$-th power. It follows that
$$\hspace{-1.5cm}\|N[\nabla u]\|^p_{L^p(\Delta_i)}\lesssim \sum_{j=i-9}^{i+5}\left[\|(f-f_{\Delta_j})\psi_j)\|^p_{\dot{L}^p_{1,1/2}}+|f_{\Delta_{j-1}}-f_{\Delta_j}|^p\right]$$
$$\qquad\qquad\qquad\qquad+\left(\sum_{j<i-9}e^{-\alpha|i-j|}\left[\|(f-f_{\Delta_j})\psi_j)\|_{\dot{L}^p_{1,1/2}}+|f_{\Delta_{j-1}}-f_{\Delta_j}|\right]\right)^p.$$
For the last term we use H\"older's inequality for series. We split the exponential into two terms
$$e^{-\alpha|i-j|}=e^{-\alpha|i-j|/p'}e^{-\alpha|i-j|/p}$$
to obtain
$$\left(\sum_{j<i-9}e^{-\alpha|i-j|}\left[\|(f-f_{\Delta_j})\psi_j)\|_{\dot{L}^p_{1,1/2}}+|f_{\Delta_{j-1}}-f_{\psi_j}|\right]\right)^p\le \left(\sum_{j<i-9} \left(e^{-\alpha|i-j|/p'}\right)^{p'}\right)^{p/p'}\times$$
$$\qquad\qquad\qquad\sum_{j<i-9}e^{-\alpha|i-j|}\left[\|(f-f_{\Delta_j})\psi_j)\|_{\dot{L}^p_{1,1/2}}+|f_{\Delta_{j-1}}-f_{\Delta_j}|\right]^p$$
$$\le C(p,\alpha)\sum_{j<i-9}e^{-\alpha|i-j|} \left[\|(f-f_{\Delta_j})\psi_j)\|^p_{\dot{L}^p_{1,1/2}}+|f_{\Delta_{j-1}}-f_{\Delta_j}|^p\right].$$

Now we sum over all indices $i\in \Z$. Since 
$\|N[\nabla u]\|^p_{L^p(\partial\Omega)}= \sum_i \|N[\nabla u]\|^p_{L^p(\Delta_i)}$ we obtain after rearranging the summation order:
\begin{equation}
\|N[\nabla u]\|^p_{L^p(\partial\Omega)}\lesssim \sum_{i\in\Z} (1+\sum_{k>9}e^{-\alpha k})\left[\|(f-f_{\Delta_i})\psi_i)\|^p_{\dot{L}^p_{1,1/2}(\partial\Omega)}+|f_{\Delta_{i-1}}-f_{\Delta_i}|^p\right]
\end{equation}
$$\qquad\qquad\lesssim \sum_{i\in\Z}[\|(f-f_{\Delta_i})\psi_i)\|^p_{\dot{L}^p_{1,1/2}(\partial\Omega)}+ \sum_{i\in\Z} |f_{\Delta_{i-1}}-f_{\Delta_i}|^p.$$

We are nearly done. Recall \refproposition{prop:Product_Rule} which gives estimate for norm of the term 
$(f-f_{\Delta_i})\psi_i$.  It follows that 
\begin{equation}\label{twot}
\|N[\nabla u]\|^p_{L^p(\partial\Omega)}\lesssim \sum_{i\in\Z} |f_{\Delta_{i-1}}-f_{\Delta_i}|^p +\|g\|^p_{L^p(\partial\Omega)}+ \|h\|^p_{L^p(\partial\Omega)},
\end{equation}
where for functions $g\in \mathcal D_s(f),\,h$ we have used the finite overlap of the balls $32\Delta_i$ which allows up to claim that
$$\sum_{i\in \Z}\|g\|^p_{L^p(32\Delta_i)}\lesssim \|g\|^p_{L^p(\partial\Omega)}.$$

For both $g$ and $h$ we have bounds on their $L^p$ norms by $\|f\|_{\dot{L}^p_{1,1/2}(\partial\Omega)}$, provided $g$ is chosen correctly. Thus \eqref{twot} only requires bound on the first term on the righthand side.\medskip

We now wish to show
\begin{align}\label{eq:fsharp}
    |f_{\Delta_{i-1}} - f_{\Delta_i}| \lesssim  \|f^\#\|_{L^p(\Tilde{\Delta}_i)},
\end{align}
where $\Tilde{\Delta}_i$ is an enlargement of $\Delta_i$ and $f^\#$ is the sharp maximal function of $f$ defined by \eqref{sharp}. Recall that we have established that $\|f^\#\|_{L^p(\partial\Omega)}\le C\|f\|_{\dot{L}^p_{1,1/2}(\partial\Omega)}$.

We have by \eqref{m22}:
$$\left|f(X,t)\chi_{\Delta_{i-1}}(X,t)-(\fint_{\Delta_i} f)\chi_{\Delta_{i-1}}(X,t) \right| \lesssim \left(f^\sharp(X,t)+\fint_Bf^\sharp\right)\chi_{\Delta_{i-1}}(X,t),$$
for $B=\Delta_{i-1}\cup\Delta_i$. Integrating over $\partial\Omega$ yields

$$|f_{\Delta_{i-1}}-f_{\Delta_i}|\le C\fint_Bf^\sharp.$$
Since the volume of $B$ is $O(1)$ the average can be replaced by $\|f^\sharp\|_{L^1(\tilde{\Delta_i})}$. From this an estimate analogous to \eqref{eq:fsharp} follows. In summary, we have for \eqref{twot} that
\begin{equation}\label{twotz}
\|N[\nabla u]\|_{L^p(\partial\Omega)}\lesssim \|f\|_{\dot{L}^p_{1,1/2}(\partial\Omega)},
\end{equation}
as desired. Hence the bounded case is also established and \reftheorem{thm:Dp_implies_Rq} holds.
\qed\medskip

\section{Interpolation and \reftheorem{thm:Rp_extrapolation}.}\label{S:interpolation}
\subsection{Atomic spaces $\dot{\HS}^{1,(\beta)}_{1,1/2}(\R^n)$.}

We plan to define natural end-point spaces for $p=1$ which interpolates well 
    with the homogeneous parabolic Sobolev spaces $\dot{L}^p_{1,1/2}(\R^n)$ defined above for $p>1$. 
It follows that we aim to define spaces that respect parabolic homogeneity. Thus for an \lq\lq atom" $a$ we expect $\nabla a$ to satisfy the usual condition defining Hardy-Sobolev space with one derivative. However in the $t$-variable we only expect to have a $1/2$-time derivative. We borrow lots of ideas from \cite{BB} (see Definition 2.16 there).

\begin{definition}\label{HSato} Let $1<\beta<\infty$. We say that a function $b:\R^n\to\R$ is a homogeneous $(1,1/2,\beta)$-atom associated to a parabolic ball $B\subset\R^n$ of radius $r=r(B)$ if 
\begin{itemize}
\item[(i)] $b$ is supported in the ball $B$,
\item[(ii)] $\|\nabla b\|_{L^\beta(\R^n)}+\|D^{1/2}_t b\|_{L^\beta(\R^n)}\le |B|^{-1/\beta'}$,
\item[(iii)] $\|b\|_{L^1(\R^n)}\le r(B)$.
\end{itemize}
Here $|B|$ denotes the usual $n$-dimensional Lebesgue measure on $\R^n$.
When $p=\infty$ we modify (ii) and replace it by $\|\nabla b\|_{L^\infty(\R^n)}+\|D^{1/2}_t b\|_{\BMO(\R^n)}\le |B|^{-1}.$
\vglue2mm

We say that a locally integrable function $f$ belongs to the homogeneous parabolic atomic Hardy-Sobolev space
$\dot{\HS}^{1,(\beta)}_{1,1/2}(\R^n)$ if there exists a family of locally integrable homogeneous $(1,1/2,\beta)$-atoms $(b_i)_i$ such that $f=\sum_i\lambda_ib_i$, with $\sum_i|\lambda_i|<\infty$. We equip the space 
$\dot{\HS}^{1,(\beta)}_{1,1/2}(\R^n)$ with the norm:
$$\|f\|_{\dot{\HS}^{1,(\beta)}_{1,1/2}}=\inf_{(\lambda_i)_i}\sum_i|\lambda_i|.$$
\end{definition}

The condition (iii) is not strictly necessary as it follows from the conditions (i) and (ii) and Poincar\'e's inequality in spatial gradient $\nabla b$.  For each fixed $t\in\R$ let $\Delta_t=B\cap \{t=const\}$ be the \lq\lq ball" in spatial variables $x\in\R^{n-1}$. Since $b$ vanishes on $\partial\Delta_t$ we have that
$$\|b\|_{L^1(\cdot,t)}=\int_{\Delta_t}|b|dx\le r\int_{\Delta_t}|\nabla b|\,dx.$$
Integrating in $t$ we then get
$$\|b\|_{L^1(\R^n)}\le r\int_{t\in\R}\int_{\Delta_t}|\nabla b|\,dx\,dt=r\int_{B}|\nabla b|\le r\|\nabla b\|_{L^\beta}|B|^{1/\beta'},$$
by H\"older. From this using (ii) we see that (iii) follows.\vglue2mm

In the definition above the condition $(ii)$ seems to be \lq\lq global" as the $L^\beta$ norm is calculated over whole space. However, clearly supp $\nabla b\subset B$. We also claim that $D^{1/2}_t b$ is \lq\lq essentially" localised to a ball $2B=B(X,2r)$ and decaying away from the support of $B$. We have the following:

\begin{lemma}\label{lhalf} For an $(1,1/2,\beta)$-atom $b$ supported on $B=B((x,t),r)$ we have for $Y=(y,s)$:
\begin{itemize}
\item[(i)] $D^{1/2}_t b(y,s)=0$ if $|x-y|\ge r$.
\item[(ii)] If $|x-y|<r$ and $|t-s|>4r^2$ then
$$|D^{1/2}_t b(y,s)|\le C|t-s|^{-3/2}\|b\|_{L^1(B\cap\{y=const\})}.$$
\end{itemize}
\end{lemma}
\begin{proof}
Using the definition of the half-derivative we have:
$$|D^{1/2}_t b(y,s)|=\left|c_1\int_{\tau\in\R}\frac{b(y,s)-b(y,\tau)}{|s-\tau|^{3/2}}d\tau\right|\le  c_1\int_{\tau\in\R}\frac{|b(y,\tau)|}{|s-\tau|^{3/2}}d\tau.$$
Given that $b$ is only supported in $B$, only $\tau\in (t-r^2,t+r^2)$ contribute to the integral above. Since  for such $\tau$: $|s-\tau|\approx |s-t|$ we get that
$$|D^{1/2}_t b(y,s)|\lesssim |t-s|^{-3/2}\|b\|_{L^1(Q\cap\{y=const\})}.$$
\end{proof}
Our aim is to prove that if $b$ is a $(1,1/2,\beta)$-atom then $b\in \dot{W}^{1}_{1,1/2}(\R^n)$ with $\dot{W}^{1}_{1,1/2}(\R^n)$ norm $\lesssim 1$. It then would follow that for every $\beta>1$ we have that
\begin{equation}\label{m22b}
\dot{\HS}^{1,(\beta)}_{1,1/2}(\R^n)\subset \dot{W}^1_{1,1/2}(\R^n)\qquad\mbox{and}\qquad \|f\|_{\dot{W}_{1,1/2}^{1}(\R^n)}\lesssim \|f\|_{\dot{\HS}^{1,(\beta)}_{1,1/2}(\R^n)}.
\end{equation}

Consider therefore and arbitrary $(1,1/2,\beta)$-atom $b$ supported on $B$ centred at $(x,t)$. By H\"older we have that
$$\|\nabla b\|_{L^1(\R^n)}+\|D^{1/2}_t b\|_{L^1(2B)}=\|\nabla b\|_{L^1(2B)}+\|D^{1/2}_t b\|_{L^1(2B)}$$
$$\le [\|\nabla b\|_{L^\beta(2B)}+\|D^{1/2}_t b\|_{L^\beta(2B)}]|2B|^{1/\beta'}.$$
Using the fact that we have a doubling measure and condition (ii) in the definition of the atom the claim
$$\|\nabla b\|_{L^1(\R^n)}+\|D^{1/2}_t b\|_{L^1(2B)}\lesssim 1$$
follows. It remains to estimate $D^{1/2}_t b$ away from the ball $2B$. For this we use \reflemma{lhalf}. By (i) and (ii):
\begin{equation}
\|D^{1/2}_t b\|_{L^1(\R^n\setminus 2B)}\le \int_{s\in\R:\, |t-s|>4r^2}\int_{|y-x|<r}|D^t_{1/2}b(y,s)|dy\,ds
\end{equation}
$$\le  \int_{s\in\R:\, |t-s|>4r^2} |t-s|^{-3/2} \int_{|y-x|<r}\|b\|_{L^1(B\cap \{ (y,\cdot) \} )}dy\, ds.$$
The interior integral in $y$ will give us $\|b\|_{L^1(B)}$ thus yielding:
\begin{equation}
\|D^{1/2}_t b\|_{L^1(\R^n\setminus 2B)}\le C\|b\|_{L^1(B)}\int_{4r^2}^\infty \tau^{-3/2}d\tau\approx r^{-1}\|b\|_{L^1(B)}\le 1,
\end{equation}
where in the final step we have used property (iii) of the definition of an atom $b$. The calculation of the $\BMO$ norm of $D^{1/2}_tb$ is analogous for $\beta=\infty$.

\begin{lemma}\label{l22} When $1<\beta\le\alpha\le\infty$ an $(1,1/2,\alpha)$-atom $b$ is a multiple of an $(1,1/2,\beta)$-atom $b$
and hence
\begin{equation}\label{m22c}
\dot{\HS}^{1,(\infty)}_{1,1/2}(\R^n)\subset \dot{\HS}^{1,(\alpha)}_{1,1/2}(\R^n)\subset \dot{\HS}^{1,(\beta)}_{1,1/2}(\R^n).
\end{equation}
Moreover, for each $\beta>1$ the space $\dot{\HS}^{1,(\beta)}_{1,1/2}(\R^n)$ is a Banach space.
\end{lemma}
\begin{proof} Only condition (ii) differs for different values of $\beta$. A calculation similar to the one we have performed above implies that 
$$\|\nabla b\|_{L^\beta(\R^n)}+\|D^{1/2}_t b\|_{L^\beta(2B)}\le [\|\nabla b\|_{L^{\alpha}(\R^n)}+\|D^{1/2}_t b\|_{L^{\alpha}(2B)}]|2B|^{1/\beta-1/\alpha}.$$
This also holds for $\alpha=\infty$ replacing the $L^\infty$ norm by $\BMO$ as $\BMO\subset L^\beta$ for $\beta<\infty$.
Assuming that $b$ is a $(1,1/2,\alpha)$-atom by (ii) of the definition of an atom an using doubling we get that
$$\|\nabla b\|_{L^\beta(\R^n)}+\|D^{1/2}_t b\|_{L^\beta(2B)}\le C_{n+1}^{1/\beta-1/\alpha}|B|^{1/\beta-1}=C_{n+1}^{1/\beta-1/\alpha} |B|^{-1/\beta'}.$$
Here $C_{n+1}$ is the doubling constant. As before $\|D^{1/2}_t b\|_{L^\beta(\R^n\setminus 2B)}$ requires a separate treatment. We again use \reflemma{lhalf}. By (i) and (ii):
\begin{equation}
\|D^{1/2}_t b\|^\beta_{L^\beta(\R^n\setminus 2B)}\le \int_{s\in\R:\, |t-s|>4r^2}\int_{|y-x|<r}|D^t_{1/2}b(y,s)|^\beta dy\,ds
\end{equation}
$$\le  \int_{s\in\R:\, |t-s|>4r^2} |t-s|^{-3\beta/2} \int_{|y-x|<r}\|b\|^\beta_{L^1(B\cap \{(y,\cdot)\})}dy\, ds.$$
By H\"older
$$\|b\|^\beta_{L^1(B\cap \{(y,\cdot)\})}\le \|b\|^\beta_{L^\beta(B\cap \{(y,\cdot)\})}r^{2(\beta-1)}.$$

Hence
\begin{equation}
\|D^{1/2}_t b\|^\beta_{L^\beta(\R^n\setminus 2B)}\le C\|b\|^\beta_{L^\beta(B)}r^{2(\beta-1)}\int_{4r^2}^\infty \tau^{-3\beta/2}d\tau =C\|b\|^\beta_{L^\beta(B)}r^{-\beta}.
\end{equation}
It follows that 
$$\|D^{1/2}_t b\|_{L^\beta(\R^n\setminus 2B)}\le Cr^{-1}\|b\|_{L^\beta(B)}.$$
As before using the (i) and (ii) we can obtain $\|b\|_{L^\beta}\le Cr\|\nabla b\|_{L^\beta(B)}$ and hence
$$\|\nabla b\|_{L^\beta(\R^n)}+\|D^{1/2}_t b\|_{L^\beta(\R^n)}\le C |B|^{-1/\beta'}.$$
From this it follows that $\frac{b}{C}$ is a $(1,1/2,\beta)$-atom. The proof that the defined space is Banach is trivial 
(c.f. \cite{BB} for similar proofs).
\end{proof}

We now consider the real interpolation between the spaces $\dot{\HS}^{1,(\beta)}_{1,1/2}(\R^n)$ and $\dot{W}^\infty_{1,1/2}(\R^n)$, where \eqref{W} defines the space $\dot{W}^\infty_{1,1/2}$. Recall som classical results on $K$-method of real interpolation from \cite{BS, BL}.

Let $A_{0}$, $A_{1}$ be  two normed vector spaces embedded in a topological Hausdorff vector space $V$. For each  $a\in A_{0}+A_{1}$ and $t>0$, we define the $K$-functional of real interpolation by
$$
K(a,t,A_{0},A_{1})=\displaystyle \inf_{a
=a_{0}+a_{1}}(\| a_{0}\|_{A_{0}}+t\|
a_{1}\|_{A_{1}}).
$$

For $0<\theta< 1$, $1\leq q\leq \infty$, we denote by $(A_{0},A_{1})_{\theta,q}$ the real interpolation space between $A_{0}$ and $A_{1}$ defined as
\begin{displaymath}
    (A_{0},A_{1})_{\theta,q}=\left\lbrace a \in A_{0}+A_{1}:\|a\|_{\theta,q}=\left(\int_{0}^{\infty}(t^{-\theta}K(a,t,A_{0},A_{1}))^{q}\,\frac{dt}{t}\right)^{\frac{1}{q}}<\infty\right\rbrace.
\end{displaymath}
It is an exact interpolation space of exponent $\theta$ between $A_{0}$ and $A_{1}$ (see \cite{BL}, Chapter II).
\begin{definition}
Let $f$  be a measurable function on a measure space $(X,\mu)$. The decreasing rearrangement of $f$ is the function $f^{*}$ defined for every $t\geq 0$ by
$$
f^{*}(t)=\inf \left\lbrace\lambda :\, \mu (\left\lbrace x:\,|f(x)|>\lambda\right\rbrace)\leq
t\right\rbrace.
$$
The maximal decreasing rearrangement of
$f$ is the function $f^{**}$ defined for every $t>0$ by
$$
f^{**}(t)=\frac{1}{t}\int_{0}^{t}f^{*}(s) ds.
$$
\end{definition}
From the properties of $f^{**}$ we mention:
\begin{itemize}
\item[1.] $(f+g)^{**}\leq f^{**}+g^{**}$.
\item[2.] $({M}f)^{*}\sim f^{**}$.
\item[3.] $\mu(\left\lbrace x;\,|f(x)|>f^{*}(t)\right\rbrace )\leq t$.
\item[4.] $\forall 1< p\leq\infty$, $\|f^{**}\|_{p}\sim\|f\|_{p}$.
\end{itemize}

We exactly know the functional $K$ for Lebesgue spaces~:
\begin{proposition}\label{KmethodL} Take $0<p_0<p_1\leq \infty$. We have~:
$$K(f,t,L^{p_0},L^{p_1}) \simeq \left(\int_0^{t^{\alpha}} \left[f^{*}(s)\right]^{p_0} ds \right)^{1/p_0} + t \left(\int_{t^{\alpha}}^\infty \left[f^{*}(s)\right]^{p_1} ds \right)^{1/p_1},$$
where $\frac{1}{\alpha}=\frac{1}{p_0}-\frac{1}{p_1}$.
\end{proposition}

When the $p_1=\infty$ and the $L^{p_1}$ endpoint is replaced by $\BMO$ we have by the result \cite{BDS}, (see also \cite{JT}) that
$$K(f,t,L^{p_0},\BMO)\approx t(M^\sharp_{p_0} f)^*(t^{p_0}),$$
where $M^\sharp_p$ is a variant of the maximal function that can be bounded by $M$, the usual maximal function.

We characterize the $K$-functional of real interpolation in the following theorem:

\begin{proposition} \label{EKHS}  \hphantom{1 pt}
\begin{itemize}
\item[1.] For all  $\beta\in(1,\infty)$, there exists $C_{1}>0$ such that for every $f\in \dot{\HS}^{1,(\beta)}_{1,1/2}(\R^n)+\dot{W}^{\infty,\BMO}_{1,1/2}(\R^n)$ and $t>0$,
$$
K(f,t, \dot{\HS}^{1,(\beta)}_{1,1/2},\dot{W}^{\infty,\BMO}_{1,1/2})\geq C_{1}t\left(|\nabla f|^{**}+|M(D^{1/2}_t f)|^{**}\right)(t);
$$
\item[2.] for $1< \beta<\infty$, there exists $C_{2}>0$ such that for every $f\in \dot{L}^\beta_{1,1/2}(\R^n)$ and $t>0$,
$$
K(f,t, \dot{\HS}^{1,(\beta)}_{1,1/2},\dot{W}^{\infty,\BMO}_{1,1/2})\leq C_{2}t\left(|\nabla f|^{\beta**\frac{1}{\beta}}+|D^{1/2}_t f|^{\beta**\frac{1}{\beta}}+|S_1(f)|^{\beta**\frac{1}{\beta}}+|f^\sharp|^{\beta**\frac{1}{\beta}}\right)(t).
$$
\end{itemize}
\end{proposition}
\begin{proof}
The lower bound is trivial and follows from the fact that $\dot{\HS}^{1,(\beta)}_{1,1/2}\subset \dot{W}^{1}_{1,1/2}$ and the characterisation of $K$ between $L^1$ and $\BMO$.

For the upper bound, the result is based on the following Calder\'on-Zygmund decomposition lemma. 

\begin{proposition}[Calder\'{o}n-Zygmund lemma for Sobolev functions]\label{CZ} 
Let $1<\beta<\infty$, $f \in L^{\beta}_{1,1/2}(\R^n)$ and $\alpha>0$. Then one can find a collection of balls $(B_{i})_{i}$, functions $b_{i}$ and a function $g\in \dot{W}^{\infty,\BMO}_{1,1/2}(\R^n)$ such that the following properties hold:
\begin{equation}
f = g+\sum_{i}b_{i}, \label{dfaze}
\end{equation}
\begin{equation}
|\nabla g(X)|\leq C\alpha,\,\quad g^\sharp(X)\le C\alpha \textrm{ and }\,\|D^{1/2}_tg\|_{BMO(\R^n)}\leq C\alpha\quad a.e.\; x\in \R^n, \label{egaze}
\end{equation}
\begin{equation}\label{ebaze}
\mbox{\rm supp } b_{i}\subset B_{i},\quad \|b_{i}\|_{\dot{\HS}^{1,(\beta)}_{1,1/2}(\R^n)}\leq C\alpha|B_{i}|,
\end{equation}
\begin{equation}\label{sbaze}
\sum_{i}|B_{i}|\leq C\alpha^{-\beta}\int (|\nabla f|+|D^{1/2}_t f|+|S_1(f)|+f^\sharp)^{\beta},
\end{equation}
\begin{equation}\label{rbaze2}
 \sum_i{\bf 1}_{B_i} \leq N {\bf 1}_{\Omega},
\end{equation}
\begin{equation}\label{rbaze}
\left\| \sum_ib_i \right\|_{\dot{\HS}^{1,(\beta)}_{1,1/2}}\leq \sum_{i}
\|b_{i}\|_{\dot{\HS}^{1,(\beta)}_{1,1/2}}\le CN\alpha|\Omega|.
\end{equation}
Here $C$ and $N$  only depend on $q$, $\beta$ and on the dimension $n$. By ${\bf 1}_A$ we denote the indicator function of the set $A$ and $\Omega=\bigcup_i B_i$.
\end{proposition}

We postpone the proof and first show how it implies the upper bound in \refproposition{EKHS}.\vglue2mm

Take $f\in \dot{L}^\beta_{1,1/2}(\R^n)$ with $\beta>1$ and $1<q\le\beta<\infty$. 
Let $t>0$. We consider the Calder\'{o}n-Zygmund decomposition of \refproposition{CZ} for $f$ with
$\alpha=\alpha(t)=\left( {M}(|\nabla f|^\beta+|D^{1/2}_t f|^\beta+|S_1(f)|^\beta+|f^\sharp|^\beta)\right)^{*\frac{1}{\beta}}(t)$.
We write $ \displaystyle f=\sum_{i}b_{i}+g=b+g $ where
$(b_{i})_{i},\,g$ satisfy the properties of the proposition. By \eqref{rbaze}
\begin{align*}
\| b \|_{\dot{\HS}^{1,(\beta)}_{1,1/2}}&\leq C\alpha(t)\sum_{i}|B_{i}|\leq C\alpha(t)\mu(\Omega),
\end{align*}
where $\Omega= \Omega_t = \cup_iB_i$.
Moreover, since $({M}f)^{*}\sim f^{**}$ and $(f+g)^{**}\leq f^{**}+g^{**}$, we get $$
\alpha(t)\lesssim \left(|\nabla f|^{\beta**{\frac{1}{\beta}}}+|D^{1/2}_t
f|^{\beta**{\frac{1}{\beta}}}+|S_1(f)|^{\beta**{\frac{1}{\beta}}}+|f^\sharp|^{\beta**{\frac{1}{\beta}}}\right)(t).
$$
Noting that for this choice of $\alpha(t)$, $|\Omega_t|\leq t$ (c.f \cite{BS},\cite{BL}), we deduce that
 \begin{align}
 K(f,t,\dot{\HS}^{1,(\beta)}_{1,1/2},\dot{W}^{\infty}_{1,1/2}) & \leq \|b\|_{\dot{\HS}^{1,(\beta)}_{1,1/2}} + t\|g\|_{\dot{W}^{\infty}_{1,1/2}} \nonumber \\
 & \lesssim t\left(|\nabla f|^{\beta**{\frac{1}{\beta}}}(t)+|D^{1/2}_t
f|^{\beta**{\frac{1}{\beta}}}(t)+|S_1(f)|^{\beta**{\frac{1}{\beta}}}(t)+|f^\sharp|^{\beta**{\frac{1}{\beta}}}(t)\right),\label{Kr}
\end{align}
for all $t>0$ and obtain the desired inequality for $f\in W^{1,\beta},\, q\leq \beta<\infty$. \end{proof}

Then integrating the $K$-functional yields

\begin{proposition} \label{IHS} For all $\beta \in(1,\infty)$ and $p\in[\beta,\infty)$, $\dot{L}^{p}_{1,1/2}(\R^n)$ is a real interpolation space between $\dot{\HS}^{1,(\beta)}_{1,1/2}(\R^n)$ and $\dot{W}^{\infty,\BMO}_{1,1/2}(\R^n)$. More precisely, we have
$$ \left(\dot{\HS}^{1,(\beta)}_{1,1/2}(\R^n),\dot{W}^{\infty,\BMO}_{1,1/2}(\R^n)\right)_{1-\frac{1}{p},p} = \dot{L}^{p}_{1,1/2}(\R^n).$$
\end{proposition}
We use the fact that the $L^p$ norm of the functions $f^\sharp$ and $S_1(f)$ are controlled by the norm of $f$ in $\dot{L}^{p}_{1,1/2}$ for all $p>1$.\medskip

\noindent {\it Proof of \refproposition{CZ}}. Let  $f\in W^{1,\beta}$, $\alpha>0$ and consider
$$\Omega=\left\lbrace X \in \R^n:  {M}(|\nabla f|^\beta+|D^{1/2}_t f|^\beta+|S_1(f)|^\beta+|f^\sharp|^\beta)(X)>\alpha^{\beta}\right\rbrace.$$ If $\Omega=\emptyset$, then set
$$
 g=f\;,\quad b_{i}=0 \, \text{ for all } i
$$
so that (\ref{egaze}) is satisfied according to the Lebesgue differentiation theorem. 
Otherwise, the maximal theorem yields
\begin{align}
    |\Omega|&\leq C\alpha^{-\beta} \big\|(|\nabla f|+ |D^{1/2}_t f|+|S_1(f)|^\beta+f^\sharp)^{\beta}\big\|_{L^1} \nonumber\\
            & \leq C \alpha^{-\beta} \Bigr(\int | \nabla f|^{\beta}  +\int |D^{1/2}_t f|^{\beta}+\int|S_1(f)|^\beta+\int|f^\sharp|^\beta \Bigl) \label{mOaze}
\\
            &<+\infty. \nonumber
\end{align}
 In particular $\Omega \neq \R^n$. Let $F$ be the complement of $\Omega$. Since $\Omega$ is an open set distinct of $\R^n$, let
$(\underline{B_{i}})$ be a Whitney decomposition of $\Omega$ (\cite{coifman1}).
That is,  the $\underline{B_{i}}$ are pairwise disjoint, and there exist two constants $C_{2}>C_{1}>1$, depending only on the metric, such that
\begin{itemize}
\item[1.] $\Omega=\cup_{i}B_{i}$ with $B_{i}=
C_{1}\underline{B_{i}}$ and the balls $B_{i}$ have the bounded overlap property;
\item[2.] $r_{i}=r(B_{i})=\frac{1}{2}d(X_{i},F)$ and $X_{i}$ is
the center of $B_{i}$;
\item[3.] each ball $\overline{B_{i}}=C_{2}B_{i}$ intersects $F$ ($C_{2}=4$ works).
\end{itemize}
For $X\in \Omega$, denote $I_{X}=\left\lbrace i:X\in B_{i}\right\rbrace$.
Recall that $\sharp I_{X} \leq N$ and fixing $j\in I_{X}$, $B_{i}\subset 7B_{j}$ for all $i\in I_{X}$. From this \eqref{rbaze2} follows. \\
Condition  (\ref{sbaze})  is satisfied due to (\ref{mOaze}). Using the doubling property, and the fact that $C_2B_i$ intersects $F$ we have
\begin{equation}\label{faze}
\int_{7B_{i}} (|\nabla f|^{\beta}+|D^{1/2}_t f|^{\beta}+|S_1(f)|^\beta+|f^\sharp|^\beta)d\mu \leq C \alpha^{\beta}|B_{i}|.
\end{equation}
Let us now define the functions $b_{i}$. For this, we construct a partition of unity.
Let $(\chi_{i})_{i}$ be  a partition of unity of $\Omega$ subordinated to the covering $(B_{i})$. Each $\chi_{i}$ is a Lipschitz function supported in $B_{i}$ with
$\displaystyle\|\,|\nabla \chi_{i}|\, \|_{\infty}\leq
{C}{r^{-1}_{i}}$ and $\displaystyle\|\,|\partial_t \chi_{i}|\, \|_{\infty}\leq
{C}{r^{-2}_{i}}$.\vglue2mm

We set $b_{i}=(f-\frac{1}{\chi_{i}(B_{i})}\int_{B_{i}}f\chi_{i})\chi_{i}$ 
    where $\chi_{i}(B_{i}) \coloneqq \int_{B_{i}}\chi_{i} \simeq |B_i|$. 
We estimate the $L^\beta$ norms of $\nabla b_i$ and $D^{1/2}_t b_i$.
We start with the $\nabla b_i$. Clearly this function is supported on $B_i$ and
$$\nabla b_i=(\nabla f)\chi_i+\left(f-\frac{1}{\chi_{i}(B_{i})}\int_{B_{i}}f\chi_{i}\right)\nabla\chi_i.$$
Using \eqref{sharp2} we have that 
$$\int_{\R^n}|\nabla b_i|^\beta\le \int_{B_i}|\nabla f|^\beta+\int_{B_i}|f^\sharp|^\beta.$$
Hence by \eqref{faze} 
$\|\nabla b_i\|_{L^\beta}\le C\alpha|B_i|^{1/\beta}$. Moving to $D^{1/2}_t b_i$ we get that
$$D^{1/2}_t b_i(x,t)=[(D^{1/2}_t f)\chi_i](x,t)+c\int_{s\in\R}\frac{\chi_i(x,t)-\chi_i(x,s)}{|s-t|^{3/2}}\left(f-\frac{1}{\chi_{i}(B_{i})}\int_{B_{i}}f\chi_{i}\right)(x,s)ds.$$
\begin{equation}
=[(D^{1/2}_t f)\chi_i](x,t)+R_i(x,t).\label{HD}
\end{equation}

We first integrate over $2B_i$. The first term enjoys an estimate similar to $\nabla b_i$ and again by \eqref{faze} 
$\|(D^{1/2}_t f)\chi_i\|_{L^\beta(\R^n)}\le C\alpha|B_i|^{1/\beta}$.
For the second term using Lipschitzness of $\chi_i$ in the $t$-variable we may estimate $\frac{|\chi_i(x,t)-\chi_i(x,s)|}{|s-t|^{3/2}}$ by $\frac{C}{r_i^2|s-t|^{1/2}}$. By \eqref{sharp2} we then have
$$\int_{2B_i}|R_i(x,t)|^\beta\le \int_{2B_i}r_i^{-\beta}\left(\int_{(x,s)\in B_i}|s-t|^{-1/2}[f^\sharp(x,s)+\fint_{B_i}f^\sharp]ds\right)^\beta\,dt\,dx.$$
Here the second term containing the average of $f^\sharp$ has an easy estimate by $C\int_{B_i}|f^\sharp|^\beta$. Hence we focus on the first term. Integration over $s$ can be split into integration over the sets where $|s-t|\approx 2^{-i}r_i^2$ for $i=0,1,2,\dots$. Hence we get a bound by 
$$\int_{2B_i}\left(\sum_{i=0}^\infty2^{-i/2}\int_{|s-t|\sim 2^{-i}r_i^2} \frac{f^\sharp(x,s)}{2^{-i}r_i^2}\right)^\beta dt\, dx.$$
We can recognise that the inside integral is bounded by a truncated maximal function $M^t(f^\sharp)(x,t)$ in $t$-variable. The maximal function is truncated as the sup runs over balls of radius $\le r_i^2$. 
It follows that the expression is bounded by 
$$\int_{2B_i}[M^t(f^\sharp)]^\beta\le \int_{3B_i}|f^\sharp|^\beta.$$
We have used boundedness of $M^t$ on $L^\beta$ and the fact that $M^t$ is truncated version of the maximal function and hence the integral on the righthand side only needs to be taken oven an enlargement of $2B_i$.
After putting all terms together we get from \eqref{faze} that
$$\|D^{1/2}_tb_i\|_{L^\beta(2B_i)}\le C\alpha|B_i|^{1/\beta}.$$

It remain to consider the away part when $(x,t)\in\R^n\setminus 2B_i$. 
There is nothing to calculate unless there exists $s$ for which $(x,s)\in B_i$ as the term $R_i(x,t)$ vanishes otherwise. Hence assuming that, we then have by \eqref{sharp2}
$$|R_i(x,t)|\le C\int_{B_i\cap \mathcal H^x}\frac{r_i}{|s-t|^{3/2}}[f^\sharp(x,s)+\textstyle\fint_{B_i}f^\sharp]ds.$$
Here $\mathcal H^x$ denotes the line $\{(x,s):s\in\R\}$. Hence
$$|R_i(x,t)|\le C\left(\frac{r_i}{d((x,t),B_i)}\right)^3\left[\fint_{B_i\cap \mathcal H^x}f^\sharp+\fint_{B_i}f^\sharp\right].$$
Here $d(x,B_i)=\inf\{d((x,t),Y):Y\in B_i\}$ measures the parabolic distance of $(x,t)$ to $B_i$.
Hence 
$$\int_{\R^n\setminus 2B_i}|R_i|^\beta\lesssim |B_i|\fint_{B_i}|f^\sharp|^\beta.$$
Now \eqref{faze} together with the previous estimates gives us:
\begin{equation}\label{bato}
\|b_i\|_{\dot{\HS}^{1,(\beta)}_{1,1/2}}=(\|\nabla b_i\|_{L^\beta(\R^n)}+
\|D^{1/2}_tb_i\|_{L^\beta(\R^n)})|B_i|^{1/\beta'}\le C\alpha|B_i|^{1/\beta+1\beta'}=C\alpha|B_i|.
\end{equation}
Hence $a_i:=b_i/(C\alpha|B_i|)$ are $(1,1/2,\beta)$-atoms and (\ref{ebaze}) is proved.
 Let $\lambda_i=C\alpha|B_i|$. It follows that 
$$b=\sum_ib_i=\sum_i \lambda_ia_i\in \dot{\HS}^{1,(\beta)}_{1,1/2}(\R^n)\mbox{ with norm bounded by: }
\sum_i |\lambda_i|=C\alpha\sum_i |B_i|.$$
where in the final step we have used finite overlap of the Whitney balls $B_i$.
Thus $\|b\|_{\dot{\HS}^{1,(\beta)}_{1,1/2}}\le CN\alpha|\Omega|$ by \eqref{rbaze2}. From this 
\eqref{rbaze} follows.

Set $\displaystyle g=f-\sum_{i}b_{i}$. Since the sum is locally finite on $\Omega$, as usual, $g$ is defined  almost everywhere on $\R^n$ and $g=f$ on $F$. Moreover, $g$ is a locally integrable function on $\R^n$.
It remains to prove (\ref{egaze}). We have
\begin{align*}
\nabla g &= \nabla f -\sum_{i}\nabla b_{i}
\\
&=\nabla f-(\sum_{i}\chi_{i})\nabla f -\sum_{i}(f-\textstyle\frac{1}{\chi_{i}(B_{i})}\textstyle\int_{B_{i}}f\chi_{i})\nabla
\chi_{i}
\\
&={\bf 1}_{F}(\nabla f) - \sum_{i}(f-\textstyle\frac{1}{\chi_{i}(B_{i})}\textstyle\int_{B_{i}}f\chi_{i})\nabla
\chi_{i}.
\end{align*}
From the definition of $F$ and the Lebesgue differentiation theorem, we have ${\bf 1}_{F}|\nabla f|\leq \alpha\;\mu -$a.e. We claim that a similar estimate holds for $$h=\sum_{i}\left(f-\frac{1}{\chi_{i}(B_{i})}\int_{B_{i}}f\chi_{i}\right)\nabla
\chi_{i},$$
 that is $|h(X)|\leq C\alpha$ for all $X\in \R^n$. For this, note first that $h$ vanishes on $F$ and the sum defining $h$ is locally finite on $\Omega$.

Then fix  $X\in \Omega$ and fix some $j\in I_X$.
Note that $\displaystyle \sum_{i}\chi_{i}(X)=1$ and $\displaystyle \sum_{i}\nabla\chi_{i}(X)=0$ and hence we can rewrite the formula for $h$ as
$$h(X)=\sum_{i\in I_X} \left[\left(\frac{1}{|7B_{j}|}\int_{7B_{j}}f \right) -\left(\frac{1}{\chi_{i}(B_{i})}\int_{B_{i}}f\chi_{i}\right)\right]\nabla
\chi_{i}(X).$$
We return to the calculation we have done above \eqref{sharp2}. 
With $B=B_i$, $\chi=\chi_i$ and any $Z\in 7B_j$ we get by it that
$$-Cr_i\left[f^\sharp(Z)+\fint_{B_i}f^\sharp\right]\le f(X)-\frac{1}{\chi_{i}(B_{i})}\int_{B_{i}}f\chi_{i}\le Cr_i\left[f^\sharp(Z)+\fint_{B_i}f^\sharp\right].$$
Hence by averaging over $7B_i$ and using that the measure is doubling we get that
\begin{equation}\label{e235}
\left|\fint_{7B_j} f-\frac{1}{\chi_{i}(B_{i})}\int_{B_{i}}f\chi_{i}\right|\le Cr_i\fint_{7B_j}f^\sharp\lesssim \alpha r_j,
\end{equation}
where we have used \eqref{faze}  for $7B_j$ in the last step. Hence
\begin{align}
|h(X)| \lesssim \sum_{i\in I_{x}}\alpha r_j r_{j}^{-1} \leq CN\alpha,
\end{align}
and the bound  (\ref{egaze}) for $\nabla g$ holds. Next, we look at $g^\sharp$. The result for $\nabla g$ already proves that
$$|g(x,t)-g(y,t)|\le C\alpha|x-y|.$$
It follows that we need to prove that
\begin{equation}
|g(x,t)-g(x,s)|\le C\alpha|t-s|^{1/2},
\end{equation}
since the desired result then follows by triangle inequality $|g(x,t)-g(y,s)|\le |g(x,t)-g(y,t)|+|g(y,t)-g(y,s)|$. Let us gain a better understanding on the function $g$. If follows from the definition of $g$ that
\begin{equation}
g(x)=\begin{cases}\label{e250}
f(x),&\quad\mbox{when }x\in F,\\
\sum_i (\chi_i(B_i)^{-1}\int_{B_i}f\chi_i)\chi_i,&\quad\mbox{when }x\in\Omega.
\end{cases}
\end{equation}
Hence if we simplify it, on the good set $g=f$ while on the bad set we take averages of $f$ on $B_i$ and then smoothly glue them together. 

It follows that if $(x,t),(x,s)\in F$ there is nothing to do. By \eqref{m1ax} we have that
$$|g(x,t)-g(x,s)|=|f(x,t)-f(x,s)|=|t-s|^{1/2}[f^\sharp(x,t)+f^\sharp(x,s)],$$
and since both points are in the set $F$, clearly $f^\sharp(x,t)+f^\sharp(x,s)\le 2\alpha$.
Next consider the case when both $(x,t),(x,s)$ belong to some neighbouring  balls $B_i$, $B_j$ with 
$B_i\subset 7B_j$ and $B_j\subset 7B_i$. It follows that  $r_i\approx r_j$ and $|t-s|\lesssim r_i^2$. Thus we are in situation where \eqref{e235} holds and we may perform the same calculation for $\partial_t g$ as we have done above for $\nabla g$. It follows that
$$\partial_tg(x,t)=\sum_{i\in I_{(x,t)}} \left[\left(\frac{1}{|7B_{j}|}\int_{7B_{j}}f \right) -\left(\frac{1}{\chi_{i}(B_{i})}\int_{B_{i}}f\chi_{i}\right)\right]\partial_t
\chi_{i}(x,t),$$
and hence $|\partial_t g|\lesssim \alpha r_i^{-1}$. Same is true at the point $(x,s)$ and in fact any point on the line segment joining these two points by the same argument. Finally by the mean value theorem
\begin{equation}
|g(x,t)-g(x,s)|\le |\partial_tg(x,\tau)||t-s|\le C\alpha |t-s|^{1/2}\frac{|t-s|^{1/2}}{r_i}.
\end{equation}
As $|t-s|^{1/2}\lesssim r_i$ the claim follows. When $(x,t)\in F$ and $(x,s)\in\Omega$ then $(x,s)\in B_i$ for some $i$ and from the construction of the Whitney balls $r_i\lesssim |t-s|^{1/2}$. Let $Y_i$ be the centre of the ball $B_i$. Previous argument gives us that $|g(Y_i)-g(x,s)|\le  C\alpha r_i\le C\alpha |t-s|^{1/2}$ and Hence if we prove 
$$|g(x,t)-g(Y_i)|=\left|f(x,t)-\frac{1}{\chi_{i}(B_{i})}\int_{B_{i}}f\chi_i\right|\lesssim  C\alpha|t-s|^{1/2},$$
then by triangle inequality we are done. Here we have used the fact that $Y_i$ being the centre of the ball $B_i$ actually belongs to $\underline{B_i}$ and these are non-overlapping and hence $\chi_i(Y_i)=1$.
A calculation identical to the done above \eqref{sharp2} yields that
\begin{equation}\label{lala}
\left|f(x,t)-\frac{1}{\chi_{i}(B_{i})}\int_{B_{i}}f\chi_i\right|\le C\sup_{Z\in B_i}d((x,t),Z)\left[f^\sharp(x,t)+\fint_{B_i}f^\sharp\right].
\end{equation}
Using \eqref{faze} and the fact that $\sup_{Z\in B_i}d((x,t),Z)\approx |s-t|^{1/2}$ we obtain our claim.

The final case is when $(x,t),(x,s)\in\Omega$ but $|t-s|^{1/2}>>\max\{r_i,r_j\}$, where $r_i,r_j$ are radii of the balls $(x,t)\in B_i$, $(x,s)\in B_j$ respectively. The argument is similar. We already know that
$|g(x,t)-g(Y_i)|\le C\alpha r_i$, $|g(x,s)-g(Y_j)|\le C\alpha r_j$ where $Y_i$, $Y_j$ are centres of their respective balls. After we average in \eqref{lala} we obtain
\begin{equation}\label{lala2}
\left|\frac{1}{\chi_{i}(B_{i})}\int_{B_{i}}f\chi_i-\frac{1}{\chi_{j}(B_{j})}\int_{B_{j}}f\chi_j\right|\le C\sup_{Z\in B_i\,Z'\in B_j}d(Z,Z')\left[\fint_{B_i}f^\sharp+\fint_{B_j}f^\sharp\right],
\end{equation}
and again by \eqref{faze} this is further bounded by $C\alpha|t-s|^{1/2}$. It follows that $g^\sharp\le C\alpha$ as desired.\vglue2mm

Finally, we consider the $\BMO$ bound for $D^{1/2}_tg$. This is clearly the hardest estimate as $D^{1/2}_t$ is a non-local operator. We aim to prove that $D^{1/2}_tg\in \BMO(\R^n)$ with norm $\lesssim\alpha$. We want to apply \reftheorem{T:equiv-new} and hence we conclude that $\D g\in\BMO(\R^n)$ instead. But as \eqref{normsbmo} holds and we already know $g$ is Lipschitz in the spatial variables with norm $\lesssim\alpha$ the conclusion about $D^{1/2}_tg$ would follow.

To prove that $\D g\in\BMO(\R^n)$ we use \reftheorem{T:equiv-new} and estimate \eqref{E:av:half-new}
for the function $g$. Let us fix a parabolic ball $B_r$. For $(x,t)\in B_r$ let us define a function 
\begin{equation}
\rho(x,t)=\begin{cases} 0,&\quad\mbox{for }(x,t)\in F=\R^n\setminus\Omega,\\
\min\{r_i=r(B_i):\, (x,t)\in B_i\},&\quad\mbox{for }(x,t)\in \Omega.\end{cases}
\end{equation}
Observe that when $(x,t)\in \Omega$ then due to the way Whitney decomposition of $\Omega$ was defined the radii of balls $B_i$ that contain $(x,t)$ are all comparable and hence for simplicity we took the smallest of such radii. It follows that we can write:

\begin{align}\nonumber
&\frac{1}{|B_r|}\int_{B_r}\int_{0}^{r}\gamma^{-3}\left|g(x,t)-\fint_{B_\gamma(x,t)}g\,\right|^2\,d\gamma\dt\dx \le\\\label{imp1}
&\quad\frac{1}{|B_r|}\sum_{i}\int_{B_i\cap B_r}\int_{0}^{\min\{r,r_i\}}\gamma^{-3}\left|g(x,t)-\fint_{B_\gamma(x,t)}g\,\right|^2\,d\gamma\dt\dx+\\&\quad\frac{1}{|B_r|}\int_{B_r}\int_{\min\{r,\rho(x,t)\}}^{r}\gamma^{-3}\left|g(x,t)-\fint_{B_\gamma(x,t)}g\,\right|^2\,d\gamma\dt\dx=I+II.\nonumber
\end{align}

We estimate terms $I$ and $II$ separately. We start with the term $I$. Clearly we only need to consider those indices $i$ such that $B_i\cap B_r\ne\emptyset$. 

Without loss of generality we may assume that each set $B_r$ is of the form $Q_r\times(t_0-r^2,t_0+r^2)$, where $Q_r$ is a ball in spatial variables.  Recall again the result of Strichartz, namely Theorem 3.3 of \cite{S80}. For a fixed $t\in\mathbb R$ we have that

\begin{equation}\label{smb12}
\frac1{|Q_r|}\int_{Q_r}\int_{|y|\le r}\frac{|g(x+y,t)+g(x-y,t)-2g(x,t)|^2}{|y|^{n+1}}dy\,dx\lesssim \|\nabla_x g(\cdot,t)\|^2_{\BMO(\R^{n-1})} .
\end{equation}
Given that we have shown that $|\nabla g(x,t)|\lesssim \alpha$ the quantity above is bounded by $C\alpha^2$.

We claim that therefore we have
\begin{align*}
&\frac1{|B_r|}\int_{B_r}\int_0^r \gamma^{-3} \left|g(x,t)-\fint_{|y|\le\gamma }g(x+y,t)dy\,\right|^2\,d\gamma\dt\dx\\
&=\frac12\frac1{|B_r|}\int_{B_r}\int_0^r \gamma^{-3} \left|\fint_{|y|\le\gamma }[g(x+y,t)+g(x-y,t)-2g(x,t)]dy\,\right|^2\,d\gamma\dt\dx\\
&\le \frac12\frac1{|B_r|}\int_{B_r}\int_0^r \gamma^{-n-2}\int_{|y|\le\gamma}|g(x+y,t)+g(x-y,t)-2g(x,t)|^2dy\,d\gamma\dt\dx
\\
&\le \frac12\frac1{|B_r|}\int_{B_r}\int_{|y|\le r}\frac{|g(x+y,t)+g(x-y,t)-2g(x,t)|^2}{|y|^{n+1}}dy\dt\dx
\\
&\le \sup_t\frac12\frac1{|Q_r|}\int_{Q_r\times\{t\}}\int_{|y|\le r}\frac{|g(x+y,t)+g(x-y,t)-2g(x,t)|^2}{|y|^{n+1}}dy\dx\lesssim\alpha^2,
\end{align*}
by \eqref{smb12}. 
To estimate the term $I$ in \eqref{imp1} by $C\alpha^2$ it suffices to establish that
$$\frac{1}{|B_r|}\sum_{i}\int_{B_i\cap B_r}\int_{0}^{\min\{r,r_i\}}\gamma^{-3}\left|\fint_{|y|\le\gamma}g(x+y,t)-\fint_{B_\gamma(x,t)}g\,\right|^2\,d\gamma\dt\dx\lesssim \alpha^2.$$

Consider any $(x,t)\in B_i\cap B_r$. By the calculations we have done above for $\partial_tg$ on $B_i$ we have that $|\partial_t g(x,t)|\lesssim \alpha r_i^{-1}$, where $r_i$ is the radius of $B_i$.  By the fundamental theorem of calculus (in $t$-variable) we see that the difference of the two averages in the term above can be bounded by
$$|\partial_t g|\gamma^2\le C\alpha r_i^{-1}\gamma^2.$$
It follows that 
\begin{align*}
&\frac{1}{|B_r|}\sum_{i}\int_{B_i\cap B_r}\int_{0}^{\min\{r,r_i\}}\gamma^{-3}\left|\fint_{|y|\le\gamma}g(x+y,t)-\fint_{B_\gamma(x,t)}g\,\right|^2\,d\gamma\dt\dx\lesssim\\
&\frac{C\alpha^2}{|B_r|}\sum_{i}\int_{B_i\cap B_r}r_i^{-2}\int_{0}^{\min\{r,r_i\}}\gamma d\gamma
\le \frac{C\alpha^2}{|B_r|}\sum_{i}|B_i\cap B_r|\frac{\min\{r,r_i\}^2}{r_i^2}\lesssim\\
& \alpha^2\sum_i\frac{|B_i\cap B_r|}{|B_r|}\le N\alpha^2,
\end{align*}
using the finite overlap of the balls $B_i$.\vglue2mm

It remains to consider the term II. We claim that for every $(x,t)\in\mathbb R^n$ we have a pointwise bound:
\begin{equation}\label{zmb1}
\int_{\min\{r,\rho(x,t)\}}^{r}\gamma^{-3}\left|g(x,t)-\fint_{B_\gamma(x,t)}f\,\right|^2\,d\gamma\lesssim \alpha^2.
\end{equation}
Consider first $(x,t)\in F=\R^n\setminus \Omega$. Then $M(S_1(f)^\beta)(x,t)\le\alpha^\beta$ and hence in particular $S_1(f)(x,t)\le\alpha$. Thus by \eqref{sqfn} the above lefthand side is bounded by $S_1(f)(x,t)^2\le \alpha^2$ using the fact that $f(x,t)=g(x,t)$. Hence the claim holds. For $(x,t)\in \Omega$ the definition of the set $\Omega$ implies that there exists a point $(y,s)\in F$ whose parabolic distance to $(x,t)$ is at most $4\rho(x,t)$. Again by \eqref{sqfn} we have that
\begin{equation}\label{zmb2}
\int_{\min\{r,\rho(x,t)\}}^{r}\gamma^{-3}\left|g(y,s)-\fint_{B_\gamma(y,s)}f\,\right|^2\,d\gamma\le\alpha^2.
\end{equation}
We need to compare \eqref{zmb1} to \eqref{zmb2}. Recall that by \eqref{split1bz} and the fact that $f^\sharp$ can be used in place of $g$ there, we get that
$$\left|\fint_{B_\gamma(y,s)}f-\fint_{B_\gamma(x,t)}f\right|\le C\rho(x,t)\left[\fint_{B_\gamma(x,t)}f^\sharp+\fint_{B_\gamma(y,s)}f^\sharp\right].$$
Since $\gamma\ge \rho(x,t)$ we get that $B_\gamma(x,t)\subset B_{4\gamma}(y,s)$. At the same time using how $\Omega$ is defined we know that
$M((f^\sharp)^\beta)(y,s)\le \alpha^\beta$. Hence by H\"older inequality we see that
$$\fint_{B_\gamma(x,t)}f^\sharp+\fint_{B_\gamma(y,s)}f^\sharp\le C \fint_{B_{4\gamma(y,s)}}f^\sharp\le C \left(\fint_{B_{4\gamma(y,s)}}|f^\sharp|^\beta\right)^{1/\beta}\le C\alpha.$$
Using the fact that $g$ itself satisfies pointwise bound $|g^\sharp|\le C\alpha$ we also have that $|g(x,t)-g(y,s)|\le C\rho(x,t)\alpha$. It follows that
\begin{equation}\nonumber
\int_{\min\{r,\rho(x,t)\}}^{r}\gamma^{-3}\left|g(x,t)-\fint_{B_\gamma(x,t)}f\,\right|^2\,d\gamma\le 4 \int_{\min\{r,\rho(x,t)\}}^{r}\gamma^{-3}\left|g(y,s)-\fint_{B_\gamma(y,s)}f\,\right|^2\,d\gamma
\end{equation}
\begin{equation}\label{zmb3}
+4C \int_{\min\{r,\rho(x,t)\}}^{r} \gamma^{-3}\rho(x,t)^2\alpha^2d\gamma.
\end{equation}
The first term on the righthand side enjoys the bound \eqref{zmb2}, while the second integral is bounded by 
$$4C\rho(x,t)^2\alpha^2[-\gamma^{-2}]_{\rho(x,t)}^\infty\lesssim \alpha^2.$$
Thus \eqref{zmb1} indeed holds. It follows that to bound the term $II$ of \eqref{imp1} by $C\alpha^2$ it suffices to establish that
\begin{equation}\label{imp2}
\frac{1}{|B_r|}\int_{B_r}\int_{\min\{r,\rho(x,t)\}}^{r}\gamma^{-3}\left|\fint_{B_\gamma(x,t)}(f-g)\,\right|^2\,d\gamma\dt\dx\le C\alpha^2.
\end{equation}
In order to understand the expression \eqref{imp2} recall first that $f=g$ outside $\Omega$. Thus we need to understand the difference of $f-g$ only on the union of the balls ${B_i}$. Furthermore, recalling \eqref{e250} we see that on $\Omega$ we have that
$f=\sum_i f\chi_i$, while $g=\sum_i(\chi_i(B_i)^{-1}\int_{B_i}f\chi_i)\chi_i$
and hence if $B_i\subset B_\gamma(x,t)$ then
$$\fint_{B_\gamma(x,t)} f\chi_i=\fint_{B_\gamma(x,t)}(\chi_i(B_i)^{-1}\textstyle\int_{B_i}f\chi_i)\chi_i,$$
It follows that the $i$-th terms of the sum defining $f$ and $g$ contribute to the average over $B_\gamma(x,t)$
equally and thus these contributions cancel. Hence, in order to understand $\fint_{B_\gamma(x,t)}(f-g)$
we only have to consider contributions of those $B_i$ for which $B_i\cap\partial B_\gamma(x,t)\ne \emptyset$.
Recalling \eqref{sharp2} we see that
\begin{equation}\label{imp3}
\left|\fint_{B_\gamma(x,t)}(f-g)\,\right|\le \frac{C}{|B_\gamma|}\sum_{i\in \{j:\,B_j\cap\partial B_\gamma(x,t)\}}
r_i\int_{B_i} f^\sharp\le \frac{C}{|B_\gamma|}\sum_{i\in \{j:\,B_j\cap\partial B_\gamma(x,t)\}}\alpha r_i|B_i|,
\end{equation} 
where in the last step we have used \eqref{faze} and the H\"older's inequality. To further simplify our notation
let us introduce the functions $A_i(x,t,\gamma)$  as follows:
\begin{equation}
A_i(x,t,\gamma)=\begin{cases}1,&\quad\mbox{if }B_i\cap \partial B_\gamma(x,t)\ne\emptyset,\\
0,&\quad\mbox{otherwise.}\end{cases}
\end{equation}
It follows that \eqref{imp2} (by using \eqref{imp3} and the fact that $|B_\gamma|\approx \gamma^{n+1}$) enjoys the estimate:
\begin{align*}
&\frac{1}{|B_r|}\int_{B_r}\int_{\min\{r,\rho(x,t)\}}^{r}\gamma^{-3}\left|\fint_{B_\gamma(x,t)}(f-g)\,\right|^2\,d\gamma\dt\dx\lesssim\\
&
\frac{\alpha^2}{|B_r|}\int_{B_r}\int_{\min\{r,\rho(x,t)\}}^{r}\gamma^{-(2n+5)}\left(\sum_i A_ir_i|B_i|\right)^2d\gamma\dt\dx=\\
&\frac{N^2\alpha^2}{|B_r|}\sum_ir_i^2|B_i|^2\int_{B_r}\int_{\min\{r,\rho(x,t)\}}^{r}A_i\gamma^{-(2n+5)}d\gamma\dt\dx\,+\\
&\frac{\alpha^2}{|B_r|}\sum_i\sum_{j\notin I(i)}r_i r_j |B_i||B_j|\int_{B_r}\int_{\min\{r,\rho(x,t)\}}^{r}A_iA_j\gamma^{-(2n+5)}d\gamma\dt\dx=\\
=&\,III+IV.
\end{align*}
Here we have introduced notation $I(i)=\{j:\,B_i\cap B_j\ne\emptyset\}$ since we claim that in the last mixed term $IV$ we only need to sum over pair of $(i,j)$ that are separated. This is due to finite overlap property, when a point $(x,t)$ belongs to multiple balls $B_i$ these balls have comparable radius and there are at most $N$ of them. Thus contribution of those terms to $\left(\sum_i A_ir_i|B_i|\right)^2$ is at most $N^2A_ir_i^2|B_i|^2$
for any such ball $B_i\ni (x,t)$. This also explains why the term $III$ contains the term $N^2$.

Let is first estimate the term $III$. Fix any $i$. If $B_r$ is such that $r\le r_i$ there is nothing to do since the integral in variable $\gamma$ is only over the set where $\gamma\ge \rho(x,t)\ge r_i$ and $\gamma\le r$. Thus this set is empty unless $r\ge r_i$. Similarly, unless $B_r\cap B_i\ne\emptyset$ there is nothing to do. It follows that
$$III\lesssim \frac{\alpha^2}{|B_r|}\sum_{i:\,B_r\cap B_i\ne\emptyset}r_i^2|B_i|^2\int_{B_r}\int_{\min\{r,\rho(x,t)\}}^{r}A_i\gamma^{-(2n+5)}d\gamma\dt\dx\le
$$
$$\lesssim \frac{\alpha^2}{|B_r|}\sum_{\{i:\,B_r\cap B_i\ne\emptyset\,\&\, r\ge r_i\}}r_i^2|B_i|^2\int_{r_i}^{\infty}\gamma^{-(2n+5)}\left(\int_{B_r}A_i\dx\dt\right) d\gamma.
$$
We consider the integral $\int_{B_r}A_i\dx\dt$. Recall that the function $A_i$ is nonzero only for those points $(x,t)$ for which the parabolic distance between the point $(x,t)$ and the centre of the ball $B_i$ is comparable to 
$\gamma\pm r_i$. The volume of such set is bounded by 
$$|B_{\gamma+r_i}(x,t)|-|B_{\gamma-r_i}(x,t)|\le C[(\gamma+r_i)^{n+1}-(\gamma-r_i)^{n+1}]\lesssim r_i\gamma^n.$$
Thus $III$ can be further estimated by
$$III\lesssim \frac{\alpha^2}{|B_r|}\sum_{\{i:\,B_r\cap B_i\ne\emptyset\,\&\, r\ge r_i\}}r_i^3|B_i|^2\int_{r_i}^{\infty}\gamma^{-(n+5)} d\gamma\lesssim \frac{\alpha^2}{|B_r|}\sum_{\{i:\,B_r\cap B_i\ne\emptyset\,\&\, r\ge r_i\}}|B_i|\frac{r_i^{n+4}}{r_i^{n+4}}.
$$
Finally, as the balls $B_i$ have finite overlap and we only count balls intersecting $B_r$ whose radius is less or comparable to $B_r$ we get that $\sum_{\{i:\,B_r\cap B_i\ne\emptyset\,\&\, r\ge r_i\}}|B_i|\lesssim |B_r|$. Hence the estimate $III\lesssim \alpha^2$ follows.

We now look at $IV$. 
Fix $i,j$. As $B_i$, $B_j$ are well separated we have that the parabolic distance between the centres of the balls is at least $2(r_i+r_j)$. Let is call this distance $2d$.  Again we only have to consider pair of indices from the set $S=\{i:\,B_r\cap B_i\ne\emptyset\,\&\, r\ge r_i\}$ for the same reasons as outlined above.

It also follows that unless $\gamma\ge d$
we cannot have in the integral defining $IV$ simultaneously $A_i=A_j=1$ as the ball $B_\gamma(x,t)$ will fail to intersect at least one of the balls. Thus
$$IV\lesssim \frac{\alpha^2}{|B_r|}\sum_{i,j\in S\,\& j\notin I(i)}r_ir_j|B_i||B_j|\int_{d}^\infty \gamma^{-(2n+5)}\left(\int_{B_r}A_iA_j\dx\dt\right)d\gamma.
$$
We crudely estimate $\int_{B_r}A_iA_j\dx\dt\lesssim \min\{r_i,r_j\}\gamma^{n}\le d\gamma^n$ since $A_iA_j\le A_i$ and then use the earlier estimate. It follows that
$$IV\lesssim \frac{\alpha^2}{|B_r|}\sum_{i,j\in S\,\& j\notin I(i)}r_ir_jd|B_i||B_j|\int_{d}^\infty \gamma^{-(n+5)}d\gamma\lesssim \frac{\alpha^2}{|B_r|}\sum_{i,j\in S\,\,\&\, d\ge r_i+r_j}\frac{r_ir_j|B_i||B_j|}{d^{n+3}}.
$$
We need to apply some simple geometric considerations to show finiteness of this sum. As now the separation condition is expressed as $d\ge r_i+r_j$ we fix $i$ and then only sum over those balls $j\in S$ for which $r_j\ge r_i$. We further split the sum to $d\approx 2^kr_i$, $k=0,1,2,\dots$. Hence
$$
IV\le \frac{\alpha^2}{|B_r|}\sum_{i\in S}r_i|B_i|\sum_{k=0}^\infty  \sum_{j\in S,\,d(B_i,B_j)\approx 2^k r_i,\,r_j\ge r_i}\frac{r_j|B_j|}{(2^kr_i)^{n+3}}.
$$
Actually, the sum in $k$ will be only over those $k$ for which $2^kr_i\le r$. In the last sum since $r_j\le d\approx 2^kr_i$ we have that
$$\sum_{j\in S,\,d(B_i,B_j)\approx 2^k r_i,\,r_j\ge r_i}\frac{r_j|B_j|}{(2^kr_i)^{n+3}}\le \sum_{j\in S,\,d(B_i,B_j)\approx 2^k r_i,\,r_j\ge r_i}\frac{|B_j|}{(2^kr_i)^{n+2}}.$$
Given the finite overlap of $B_j$ the sum of volumes of those $B_j$ whose distance to $B_i$ is approximately $2^kr_i$ can be at most multiple of the volume of the parabolic ball of radius $2^kr_i$. Hence
$$\sum_{j\in S,\,d(B_i,B_j)\approx 2^k r_i,\,r_j\ge r_i}\frac{r_j|B_j|}{(2^kr_i)^{n+3}}\le \frac1{2^kr_i}.$$
It follows that
$$IV\lesssim \frac{\alpha^2}{|B_r|}\sum_{i\in S}r_i|B_i|\sum_{k=0}^\infty (2^kr_i)^{-1}\le \frac{2\alpha^2}{|B_r|}\sum_{i\in S}|B_i|\le 2N\alpha^2.$$
From this by \eqref{imp1} and combining the estimates for $I$ and $II$ we get that
\begin{align}\nonumber
&\frac{1}{|B_r|}\int_{B_r}\int_{0}^{r}\gamma^{-3}\left|g(x,t)-\fint_{B_\gamma(x,t)}g\,\right|^2\,d\gamma\dt\dx \lesssim \alpha^2.
\end{align}
Hence $\D g\in\BMO(\R^n)$ and therefore $D_t^{1/2}g\in \BMO(\R^n)$ as desired.
\qed

\begin{lemma}\label{l27} Fix $1<\beta\le \infty$. There exists a large constant $M_\beta>1$ such that any 
$(1,1/2,\beta)$-atom $a$ can be decomposed as 
\begin{equation}\label{at1}
a=M_\beta a^\infty +\sum_i \lambda_i a_i,
\end{equation}
such that $a^\infty$ is an $(1,1/2,\infty)$-atom, all $a_i$ are $(1,1/2,\beta)$-atoms and
$\sum_i|\lambda_i|\le 1/2$.\vglue1mm

Hence for all $k\ge 1$, $a$ can be written as 
\begin{equation}\label{at2}
a=M_\beta \sum_i \mu_ia_i^\infty+\sum_i \lambda_i a_i,
\end{equation}
such that that all $a_i^\infty$ are $(1,1/2,\infty)$-atoms, all $a_i$ are $(1,1/2,\beta)$-atoms, 
$\sum_i|\mu_i|\le 2-2^{-k+1}$ and $\sum_i|\lambda_i|\le 2^{-k}$.
\end{lemma}
\begin{proof} Clearly, the second claim follows from the first by induction, since each $a_i$ atom in \eqref{at1}
can be again decomposed as a sum of $(1,1/2,\infty)$-atom and infinite sum of $(1,1/2,\beta)$-atoms. By relabelling the atoms, if necessary, we then obtain  \eqref{at2} for $k=2$ and inductively for any larger $k$.\vglue1mm

When $\beta=\infty$ there is nothing to prove. Suppose therefore that $1<\beta<\infty$.
To prove \eqref{at1} we use the \refproposition{CZ} and apply it to the $(1,1/2,\beta)$-atom $a$. Since $a\in L^\beta_{1,1/2}(\R^n)$ we see that for some constant $C_1$ we have that
$$\int (|\nabla a|+|D^{1/2}_t a|+|S_1(a)|+a^\sharp)^{\beta}\le C_1\|a\|_{L^\beta_{1,1/2}}^\beta=C_1(\|\nabla a\|_{L^\beta}+\|D^{1/2}_t a\|_{L^\beta})^\beta$$
$$\le C_2|B|^{-\beta/\beta'}=C_2|B|^{1-\beta}.$$

Assume that the atom $a$ is supported in a parabolic ball $B$ or radius $r(B)$. With the constants $C$ and $N$ as in \eqref{sbaze} and \eqref{rbaze} we let $\alpha=|B|^{-1}(2C^2C_2N)^{-1/(\beta-1)}$ and decompose our atom $a$
as the sum $a=g+\sum b_i$ using \refproposition{CZ}. 

Recall the calculation for each $b_i$ we have done previously in \eqref{bato}. It follows from it that 
$$b=\sum_ib_i=\sum_i\lambda_ia_i, \qquad\mbox{where each $a_i$ is an $(1,1/2,\beta)$ atom and}$$
$$\sum_i|\lambda_i|=C\alpha\sum_i|B_i|\le C^2N\alpha^{1-\beta} \int (|\nabla a|+|D^{1/2}_t a|+|S_1(a)|+a^\sharp)^{\beta}\le C^2C_2N\alpha^{1-\beta}|B|^{1-\beta}.$$
by  \eqref{sbaze} and calculation above for done for $a$. It follows that a choice of $\alpha=|B|^{-1}(2C^2C_2N)^{-1/(\beta-1)}$ makes the righthand side $\le 1/2$ as desired. It remains to show that $g$
can be seen as a multiple of a $(1,1/2,\infty)$ atom.

Recall that it follows from the proof of \refproposition{CZ} that $g=a$ on the set where
$${M}(|\nabla a|^\beta+|D^{1/2}_t a|^\beta+|S_1(a)|^\beta+|a^\sharp|^\beta)(X)\le \alpha^{\beta}.$$
In particular consider $X\in\R^n$ such that the parabolic distance of $X$ to the center of the ball $B$ is at least 
$C_3r(B)$ for some $C_3>1$ to be determined.

Given that $a$ is supported in $B$ and the distance of $X$ to the closest point of $B$ is at least $C_3-1$ we see that
$$M(|\nabla a|^\beta)(X)\le |B_{(C_3-1)r(B)}(X)|^{-1}\int_{\R^n}|\nabla a|^\beta\le K|B|^{-1}(C_3-1)^{-(n+1)}
|B|^{-\beta/\beta'}
$$
$$\le K(C_3-1)^{-(n+1)}|B|^{-\beta}\le \frac{\alpha^\beta}4=|B|^{-\beta}(2C^2C_2N)^{-\beta/(\beta-1)}$$
will hold, provided $C_3$ is chosen sufficiently large (and the choice does not depend on $B$). 
Similar calculation can be done for $D^{1/2}_t a$ taking into account the decay of the half-derivative away from the support of $a$ as seen in \reflemma{lhalf}, as well as for $S_1(a)$ where by the formula 
\eqref{sqfn} we see that
$$S_{1}(a)(X)\le \left(\int_{(C_3-1)r(B)}^\infty \alpha^{-2n-5}\|a\|_{L^1}^2d\alpha\right)^{1/2}\le
K\|a\|_{L^1}r(B)^{-1}|B|^{-1}(C_3-1)^{-(n+2)}.$$
Using (iii) of the definition of an atom this is further bounded by $K'|B|^{-1}(C_3-1)^{-(n+2)}$.
Hence once again for sufficiently large $C_3>1$ independent of $B$ we will have $S_1(a)^\beta(X)\le \frac{\alpha^
\beta}{4}$ Analogous calculation works for $a^\sharp$. Hence we conclude that for sufficiently large $C_3>1$ we have $g(X)=a(X)=0$ for $X\notin C_3B$, where $C_3B$ is the enlargement the original $B$. Thus $g$ has support in  $C_3B$.

Given this, we need to check the condition (ii) from the definition of an $(1,1/2,\infty)$-atom. By \refproposition{CZ} we have that $|\nabla g|\le C\alpha$ and $\|D^{1/2}_t g\|_{\BMO}\le C\alpha$, and hence
$$\|\nabla g\|_{L^\infty(\R^n)}+\|D^{1/2}_t g\|_{\BMO(\R^n)}\le C(2C^2C_2N)^{-1/(\beta-1)}|B|^{-1}$$
$$\le
C(2C^2C_2N)^{-1/(\beta-1)}C_3^{n+1}|C_3B|^{-1}.
$$
This mean that if we set $M_\beta=C(2C^2C_2N)^{-1/(\beta-1)}C_3^{n+1}$ then $a^\infty=\frac{g}{M_\beta}$ is 
 an $(1,1/2,\infty)$-atom with respect to the ball $C_3B$. From this our claim follows.
\end{proof}

Lemma above implies that we may write any $(1,1/2,\beta)$-atom $a$ as an infinite sum $\sum_i\lambda_i a^\infty_i$ of $(1,1/2,\infty)$-atoms $a^\infty_i$ with $\sum_i|\lambda_i|\le 2M_\beta$, where the equality $a=\sum_i\lambda_i a^\infty_i$ is understood to hold in the space $\dot{W}^1_{1,1/2}(\R^n)$.
Hence we have the following as a consequence of \reflemma{l22} and \reflemma{l27}.

\begin{proposition} For all $1<\beta\le\infty$ the Banach spaces $\dot{\HS}^{1,(\beta)}_{1,1/2}(\R^n)$ are equal to each other with norms that are equivalent. Thus we may drop the index $\beta$ and denote this space unambiguously by $\dot{\HS}^{1}_{1,1/2}(\R^n)$.
\end{proposition}

Thus as a consequence of this, \refproposition{IHS} and re-interpolation theorem \cite{BL} we have that

\begin{proposition} \label{IHS2} For all $1<p<q<\infty$ and $\theta\in (0,1)$ satisfying $\frac1p=(1-\theta)+\frac\theta{q}$,  $\dot{L}^{p}_{1,1/2}(\R^n)$ is a real interpolation space between $\dot{\HS}^{1}_{1,1/2}(\R^n)$ and $\dot{L}^{q}_{1,1/2}(\R^n)$. More precisely, we have
$$ \left(\dot{\HS}^{1}_{1,1/2}(\R^n),\dot{L}^{q}_{1,1/2}(\R^n)\right)_{\theta,p} = \dot{L}^{p}_{1,1/2}(\R^n),$$
with equivalent norms. 
\end{proposition}

\begin{corollary} Let $\Omega=\mathcal O\times \R$, where $\mathcal O\subset\R^n$ is a bounded or unbounded Lipschitz domain. Then the conclusions of the above theorem remain to be true, that is
for all $1<p<q<\infty$ and $\theta\in (0,1)$ satisfying $\frac1p=(1-\theta)+\frac\theta{q}$,  $\dot{L}^{p}_{1,1/2}(\partial\Omega)$ is a real interpolation space between $\dot{\HS}^{1}_{1,1/2}(\partial\Omega)$ and $\dot{L}^{q}_{1,1/2}(\partial\Omega)$. More precisely, we have
$$ \left(\dot{\HS}^{1}_{1,1/2}(\partial\Omega),\dot{L}^{q}_{1,1/2}(\partial\Omega)\right)_{\theta,p} = \dot{L}^{p}_{1,1/2}(\partial\Omega),$$
with equivalent norms. 
\end{corollary}

This follows in the unbounded case directly from \eqref{defsp}, 
    and in the bounded case using a partition of unity by splitting $\partial \mathcal O$ to local coordinate patches.

\subsection{Proof of \reftheorem{thm:Rp_extrapolation}}
If we can prove the endpoint bound
\begin{align}\label{eq:Endpoint}
    \| N[\nabla u] \|_{L^1(\partial\Omega)} \lesssim \| f \|_{\dot{\HS}^{1}_{1,1/2}(\partial\Omega)},
\end{align}
    then \reftheorem{thm:Rp_extrapolation} follows by real interpolation since the operator $f\mapsto N[\nabla u]$ is sublinear. We follow the same approach as in the  elliptic case outlined in \cite{DK}.
It suffices to prove $ \| N[\nabla u] \|_{L^1(\partial\Omega)} \le C$ for \( f= b \), 
    where \( b \) is a \( (1,1/2,\infty) \)-atom supported on a boundary ball 
    \( \Delta_0 = \Delta_R(P_0,t_0) \).
Note that \( b = ca \), where \( a \) is a \( (1,1/2,p) \)-atom. 
    Hence, by the solvability of \( (R_L)_p \) we have that,
	\begin{align*}
		\| N[\nabla u] \|_{L^1(2\Delta_0)} &\leq |\Delta_0|^{1/p'} \| N[\nabla u] \|_{L^p(\partial\Omega)}
		\\
		&\lesssim |\Delta_0|^{1/p'} \big( \| \nabla a \|_{L^p(\partial\Omega)} + \| D_t^{1/2} a \|_{L^p(\partial\Omega)} \big)
		\lesssim 1. 	
	\end{align*}
Thus it remains to show that
\begin{align}\label{eq:Atom away bound}
	\| N[\nabla u] \|_{L^1(\partial\Omega\setminus 2\Delta_0)} \lesssim 1. 	
\end{align}

By the fundamental theorem of calculus
\[ |b| \lesssim R \| \nabla b \|_{L^\infty} \lesssim R^{2-n}. \]

Without loss of generality we may assume that $b\ge 0$ since by splitting $b$ into positive and negative part we shall still have the $L^\infty$ bound of both parts by $R^{2-n}$, which is the only property of $b$ we will use below.

Next, we write \( \partial\Omega \setminus 2\Delta_0 = \cup_{j\geq 3} A_j \),
    where
\[ A_j \coloneqq \{ (P,s) \in \partial\Omega : \|(P,s) - (P_0,t_0)\| \approx 2^jR \}. \]
The sets $A_j$ are annuli  that can be covered by boundary balls \( \Delta_j^i \) of finite overlap 
    such that \( r(\Delta_j^i) \approx 2^jR \), and on the enlargement  \( \Tilde{\Delta}_j^i  \) of $\Delta_j^i$ we have $u=0$. 
This cover requires a finite fixed number of balls (that only depend on dimension $n$). 
 
By the proof of \reftheorem{thm:Reverse Holder}
\begin{equation}\label{truncN}
\fint_{\Delta_j^i} {N}^{2^jR}[\nabla u]
    \lesssim (2^jR)^{-1} \Big( \fint_{T(4\Delta_j^i)} u^2 \Big)^{1/2}
    \lesssim (2^jR)^{-1} u(V^+(\Delta^i_j)),
\end{equation}    
where $V^+(\Delta^i_j)$ is time-forward corkscrew point of the ball  $4\Delta_j^i$.   Give the positions of the balls $4\Delta_j^i$ the $u(V^+(\Delta^i_j))$ can further be controlled by 
Proposition \ref{prop:HarnackInequality} which implies that we have
$u(V^+(\Delta^i_j))\lesssim u(P_0, 2^jR, t_0+100(2^jR)^2)$.
By  \eqref{eq:GreenToMeas} and \reflemma{lemma:ParabolicGrwothEstimateForGwithalpha}
\begin{align*}
    u(P_0, 2^jR, t_0+100(2^jR)^2) 
    &\lesssim R^{2-n} \omega^{(P_0, 2^jR, t_0+100(2^jR)^2)}(\Delta_0)
    \\
    &\lesssim G(P_0, 2^jR, t_0+100(2^jR)^2,V^+(\Delta_0))
    \lesssim  R^\alpha (2^jR)^{-(n+\alpha)}.
\end{align*}
Using this in \eqref{truncN} implies that 
\begin{equation}\label{truncN2}
\fint_{\Delta_j^i} {N}^{2^jR}[\nabla u]
    \lesssim (2^jR)^{-1} R^\alpha (2^jR)^{-(n+\alpha)}\lesssim 2^{-j\alpha}(2^jR)^{-n-1}.
\end{equation}    

Thus 
\[ \int_{A_j} {N}^{2^jR}[\nabla u]
    \leq \sum_i (2^jR)^{n+1} \fint_{\Delta_j^i} {N}^{2^jR}[\nabla u]
    \lesssim 2^{-j\alpha}. \]
So far we have an estimate for the truncated non-tangential maximal function at the height $2^jR$ for $A_j$.
Consider any point $(X,t)\in\Gamma (S,\tau)$ for $(S,\tau)\in A_j$ and $\delta(X,t)\geq 2^jR$.
By  Caccioppolli and interior Harnack
\[ \Big( \fint_{C_{\delta(X,t)/4}(X,t)} |\nabla u|^2 \Big)^{1/2}
    \lesssim (2^jR)^{-1}\Big( \fint_{C_{\delta(X,t)/2}(X,t)} |u|^2 \Big)^{1/2}\]
\[    \lesssim  (2^jR)^{-1}u(P_0,\delta(X,t),t_0+100\delta(X,t)^2). \]
This can be further estimated by \eqref{eq:GreenToMeas} and \reflemma{lemma:ParabolicGrwothEstimateForGwithalpha} and taking a supremum over all points with 
$\delta(X,t)\geq 2^jR$ again yields a bound by $2^{-j\alpha}(2^jR)^{-n-1}$.

Hence by combining these two estimates we obtain that
\[ \int_{A_j} {N}[\nabla u]
    \leq \int_{A_j} \Tilde{N}^{2^jR}[\nabla u]
    + |A_j| (2^{-j\alpha}(2^jR)^{-n-1}
    \lesssim 2^{-j\alpha}. \]
  Summing over all $i\ge 3$ yields that
\[ \int_{\partial\Omega \setminus 2\Delta_0} {N}[\nabla u]
    = \sum_{j\geq 3} \int_{A_j} {N}[\nabla u]
    \lesssim \sum_{j\geq 3} 2^{-j\alpha} \lesssim_{\alpha} 1. \]

In the proof above we have used \reftheorem{thm:Reverse Holder} which implies that the argument given above is fully correct only if the domain $\mathcal O$ is unbounded. If $\mathcal O$ is bounded we can only consider sets $A_j$ up to the size of comparable to the diameter of $\mathcal O$, that is for $j\in\mathbb N$ such that $2^jR\lesssim \mbox{diam}(\mathcal O)$. Say $d=\mbox{diam}(\mathcal O)$ and let $\Delta_{10d}=\Delta_{10d}(P_0,t_0)\subset\partial\Omega$. Then the calculation above yields the estimate
\[ \int_{\Delta_{10d} \setminus 2\Delta_0} {N}[\nabla u]
    \lesssim \sum_{j\geq 3}^{[\log_2(R^{-1}d)]+4} 2^{-j\alpha} \lesssim_{\alpha} 1. \]
The set of boundary points with distance more than $10d$ to $(P_0,t_0)$  requires a separate argument.
However, for such points \reflemma{lemma:Exponential Decay} does apply implying an exponential decay in the $L^1$ norm of ${N}[\nabla u]$ and thus obvious summability. Therefore the claim as stated also holds in the case $\mathcal O$ is bounded. \qed\medskip

We record the statement \eqref{eq:Atom away bound} more formally as it is interesting on  its own right and highlights what is the main obstacle preventing us from establishing solvability of the Regularity problem in the end-point $p=1$ case. Compare this to Theorem 3.9 of \cite{DK} which states a similar result in the elliptic case.

\begin{theorem}\label{ZeroPart} Assume that $\omega\in A_{\infty}(d\sigma)$, where $\omega$ is the parabolic measure for the operator $L^*=\partial_t+\div(A^T\nabla\cdot)$ on a Lipschitz cylinder $\Omega=\mathcal O\times \R$.
Then for any $p\in (1,\infty]$ the following holds. There exists $C(p)>0$ such that for any 
homogeneous $(1,1/2,p)$-atom associated to a parabolic ball $\Delta\subset \partial\Omega$
(c. f. Definition \ref{HSato}),
 the weak solution $u$ of the equation $Lu=-\partial_tu+\div(A\nabla u)=0$ with boundary datum $f$ satisfies the estimate
 \begin{equation}\label{eqSato}
\|N[\nabla u]\|_{L^1(\partial\Omega\backslash 2\Delta)}\leq C(p).
\end{equation}
\end{theorem}

\begin{proof} The proof as given above holds for $(1,1/2,\infty)$-atoms. It uses boundedness of such atom and thus does not apply to $(1,1/2,p)$-atoms for $p<\infty$. But we may assume that $f\ge 0$ as it can be split as the sum of $f^+-f^-$ with each part being modulo a constant again an atom.
Thanks to Lemma \ref{l27} we know that any $(1,1/2,p)$-atom $f$, $p>1$ can be written as as sum 
$$f=\sum_i \mu_ia_i^\infty,$$ 
where $a_i^\infty$ are  $(1,1/2,\infty)$-atoms and $\sum_i |\mu_i|\le M_p$ for some large constant $M_p>1$
independent of $f$. Furthermore, thinking about the support of the atoms $a_i^\infty$ the construction we gave in the proof of Lemma \ref{l27} implies that that $a_i^\infty$ will be supported in some enlargement of $\Delta$,
say $m\Delta$ for some $m>1$ which depends of the choice of $\alpha$ in the proof of Lemma \ref{l27}.
Thus if we do not make any changes in the argument given we will obtain for each $a_i^\infty$ the estimate
\eqref{eq:Atom away bound}, i.e.:
\begin{align}\label{eq:Atom away bound2}
	\| N[\nabla u_i] \|_{L^1(\partial\Omega\setminus 2m\Delta)} \le K,
\end{align}
making the constant in the estimate explicit. Here $u_i$ solves our PDE with boundary data $a_i^\infty$. It follows that
\begin{align}\label{eq:Atom away bound3}
	\| N[\nabla u] \|_{L^1(\partial\Omega\setminus 2m\Delta)} \le \sum_i|\mu_i|\| N[\nabla u_i] \|_{L^1(\partial\Omega\setminus 2m\Delta)}\le M_pK.
\end{align}
This is not quite \eqref{eqSato} as we have removed a larger set from $\partial\Omega$, namely $2m\Delta$
instead of just $2\Delta$ as claimed. However, this is only a small technical issue. Thinking about the set
$2m\Delta\setminus 2\Delta$ we are in the regime where $f$ vanishes, it actually vanishes on an enlargement of this set as well as it clearly vanishes on $4m\Delta\setminus \Delta$.

Hence we get that Proposition \ref{prop:Boundary Cacciopolli} applies as well as the boundary H\"older continuity \eqref{eq:BoundaryHoelderNonnegativeCorkscrewEstimate}. It follows that
$u$ is bounded on the union of non-tangential cones 
$$\mathcal S:=\bigcup_{(q,\tau)\in 3m\Delta\setminus (3/2)\Delta}\Gamma(q,\tau),$$
where the bound for the near part (i.e. points $(X,t)\in\mathcal S$ with $\delta(X,t)\lesssim\mbox{diam}(\Delta)=:R$) is implied by 
\eqref{eq:BoundaryHoelderNonnegativeCorkscrewEstimate} which requires bound on value of $u$
at some corkscrew point $V^+$ of $5m\Delta$ but that bound is implied by The boundary Poincar\'e inequality
(Proposition \ref{prop:Boundary Poincare}) and the bound \eqref{eq:Atom away bound3} as the point $V^+$ is in a region covered by the estimates for $N[\nabla u]$ on the set $\partial\Omega\setminus 2m\Delta$.
The bound for the away part of the set $\mathcal S$  is easier and follows directly from \eqref{eq:Atom away bound3}.

Hence we can claim for all $(X,t)\in\mathcal S$ that $|u|\lesssim R(R^{-n-1}) \| N[\nabla u] \|_{L^1(\partial\Omega\setminus 2m\Delta)}=R^{-n}M_pK$. Here the normalization factor $R^{-n-1}$ comes from the fact that $R^{n+1}$ is comparable to the surface area of $m\Delta$. 

The final piece of the puzzle is Theorem \ref{thm:Reverse Holder} which we apply on balls $B_i$ that cover the boundary $2m\Delta\setminus 2\Delta$ of radius comparable to $R$ such that $4B_i\subset 3m\Delta\setminus (3/2)\Delta$ and balls $B_i$ have finite overlap.  Thanks to the bound for $|u|$ on $\mathcal S$ we have by \eqref{zsaz} on each $B_i$
$$     \fint_{B_i} {N}[\nabla u]
        \lesssim \fiint_{T(2B_i)} |\nabla u|\lesssim R^{-1}\left(\fiint_{T(3B_i)} |u|^2\right)^{1/2}\lesssim
        R^{-1}R^{-n}M_pK=R^{-n-1}M_pK , $$
where in the second step we have used the boundary Caccioppoli inequality (Proposition \ref{prop:Boundary Cacciopolli}).        
Equivalently, $\int_{B_i} {N}[\nabla u] \lesssim M_pK$ and this after summing over all $i$ we get that
$$
	\| N[\nabla u] \|_{L^1(2m\Delta\setminus 2\Delta)} \le C M_pK.
$$
From \eqref{eq:Atom away bound3} and the above estimate \eqref{eqSato} follows.
\end{proof}


\end{document}